\documentclass[11pt,reqno]{amsart}
\usepackage{fullpage}
\usepackage{mathrsfs,amssymb,graphicx,verbatim,amsmath,amsfonts}
\usepackage{paralist}
\usepackage[breaklinks,pdfstartview=FitH]{hyperref}
\usepackage{upgreek}
\usepackage{tikz}
\usepackage[mathscr]{euscript}
\usepackage{relsize}
 \usepackage{amsaddr}

\hypersetup{ocgcolorlinks=true,allcolors=testc}
\hypersetup{
     colorlinks   = true,
     citecolor    = black
}
\hypersetup{linkcolor=black}
\hypersetup{urlcolor=black}

\renewcommand{\leq}{\leqslant}
\renewcommand{\geq}{\geqslant}

\renewcommand{\gamma}{\upgamma}
\allowdisplaybreaks

\usepackage{esint}
\usepackage[autostyle]{csquotes}

\renewcommand{\pi}{\uppi}

\newcommand{\e}{\varepsilon}
\newcommand{\R}{\mathbb R}

\newtheorem{theorem}{Theorem}
\newtheorem{lemma}[theorem]{Lemma}
\newtheorem{proposition}[theorem]{Proposition}

\theoremstyle{remark}
\newtheorem{remark}[theorem]{Remark}

\newcounter{quest}
\newtheorem{question}[quest]{Question}

\renewcommand{\tau}{\uptau}

\renewcommand{\xi}{\upxi}
\renewcommand{\rho}{\uprho}

\newcommand{\N}{\mathbb N}

\newcommand{\eqdef}{\stackrel{\mathrm{def}}{=}}

\renewcommand{\theta}{\uptheta}
\renewcommand{\lambda}{\uplambda}

\renewcommand{\gamma}{\upgamma}
\renewcommand{\beta}{\upbeta}
\renewcommand{\alpha}{\upalpha}
\renewcommand{\kappa}{\upkappa}
\renewcommand{\psi}{\uppsi}
\renewcommand{\rho}{\uprho}
\renewcommand{\delta}{\updelta}
\renewcommand{\pi}{\uppi}
\renewcommand{\omega}{\upomega}
\renewcommand{\sigma}{\upsigma}

\renewcommand{\eta}{\upeta}

\renewcommand{\kappa}{\upkappa}
\renewcommand{\mu}{\upmu}
\renewcommand{\nu}{\upnu}
\renewcommand{\pi}{\uppi}
\renewcommand{\zeta}{\upzeta}

\newcommand{\mb}{\mathbb}

\newcommand{\ms}{\mathscr}
\newcommand{\msf}{\mathsf}
\newcommand{\mr}{\mathrm}

\begin{document}

\title{Cutoff for the Bernoulli--Laplace urn model with $o(n)$ swaps}
\author{Alexandros Eskenazis and Evita Nestoridi}
\address{Mathematics Department, Princeton University, Princeton, NJ 08544-1000, USA}
\thanks{{\it E-mail addresses:} \href{mailto:ae3@math.princeton.edu}{\nolinkurl{ae3@math.princeton.edu}}, \href{mailto:exn@princeton.edu}{\nolinkurl{exn@princeton.edu}}.
}
\thanks{This work was completed while the first named author was in residence at the Institute for Pure \& Applied Mathematics at UCLA for the long program on Quantitative Linear Algebra. He would like to thank the organizers of the program for the excellent working conditions. He was also supported in part by the Simons Foundation.}



\begin{abstract}
We study the mixing time of the $(n,k)$ Bernoulli--Laplace urn model, where $k\in\{0,1,\ldots,n\}$. Consider two urns, each containing $n$ balls, so that when combined they have precisely $n$ red balls and $n$ white balls. At each step of the process choose uniformly at random $k$ balls from the left urn and $k$ balls from the right urn and switch them simultaneously. We show that if $k=o(n)$, this Markov chain exhibits mixing time cutoff at $\frac{n}{4k}\log n$ and window of the order $\frac{n}{k}\log\log n$. This is an extension of a classical theorem of Diaconis and Shahshahani who treated the case $k=1$.
\end{abstract}

\maketitle

{\footnotesize
\noindent {\em 2010 Mathematics Subject Classification.} Primary: 60J10; Secondary: 60C05, 60G42.

\noindent {\em Key words and phrases.} Markov chain, mixing time, Bernoulli--Laplace urn model, cutoff phenomenon.}

\section{Introduction}
Mixing time of card shuffling schemes and combinatorial urn models is a widely studied subject in the discrete probability literature. Recently, there has been a lot of interest for shuffling large decks of cards and shuffling models that are actually used in real life (see, for instance, \cite{BayerDiac}, \cite{AssafDiaSound}, \cite{NW} and \cite{DiaPal}). This paper focuses on a specific shuffle of a deck of cards that is practiced by casinos (see \cite{NW}) and can be described as follows. Fix $n\in\N$ and $k\in\{0,1,\ldots,n\}$. Initially, we have an unshuffled deck of $2n$ cards. At each step of the process, we cut the deck in two piles of equal size, shuffle each pile independently and perfectly and then reassemble the deck. Finally, we move the top $k$ cards of the reassembled deck to the bottom and repeat the process.

As explained in \cite{NW}, this card shuffling scheme is in one-to-one correspondence with the classical Bernoulli--Laplace urn model with parameters $(n,k)$, which is defined as follows. Initially, we have two urns, each containing $n$ balls, so that when combined they have precisely $n$ red balls and $n$ white balls. At each step of the process, we pick $k$ balls from each urn uniformly at random and switch them simultaneously. 

Let $X_t$ denote the number of red balls in the left urn at time $t$. For $x, y \in \mathcal{X}= \{0,\ldots,n\}$, let $P^t_{x}(y)$ be the probability of having $y$ red balls in the left urn after $t$ steps, given that the left urn initially contained $x$ red balls, that is
\begin{equation}
P_x^t(y) = \mb{P} \big( X_t = y \ \big| \ X_0=x\big).
\end{equation}
It has been proven by Ta\"\i bi \cite{Taibi} that the sequence of probability measures $P^t_{x}$ converges to the hypergeometric distribution
\begin{equation} \label{eq:pin}
\pi_n(j)= \frac{{n \choose j} {n \choose n-j}}{ {2n \choose n}},\quad 0 \leq j \leq n
\end{equation}
as $t \rightarrow \infty$, with respect to the total variation distance
\begin{equation}
d(t)\eqdef \max_{x\in\mathcal{X}} \|P^t_{x}-\pi_n\|_{\mr{TV}} = \frac{1}{2} \max_{x\in\mathcal{X}} \sum_{y \in \mathcal{X}} \vert P^t_{x}(y) - \pi_n(y) \vert.
\end{equation}
A question which arises naturally, is to determine the rate of convergence to stationarity of the above process, which is quantified by the mixing time
\begin{equation}
t_{\mr{mix}}(\varepsilon)= \min \big\{t \in \mathbb{N}: \max_{x \in \mathcal{X}} \|P^t_x-\pi_n\|_{\mr{TV}}\leq \varepsilon\big\}.
\end{equation}
The mixing time of the Bernoulli--Laplace model was first studied by Diaconis and Shahshahani in \cite{Bernoulli-Laplace}, who proved that for $k=1$ the chain exhibits cutoff at $\frac{n}{4} \log n$ with window $n$. Afterwards, Donnelly, Floyd and Sudbury \cite{DLS} showed that the separation distance mixing time of the Bernoulli--Laplace urn model for $k=1$ exhibits cutoff as well and Belsley \cite{Eric} proved cutoff for the mixing time of distance regular graphs, which is a generalization of \cite{Bernoulli-Laplace}. The mixing time of the $(n,k)$ Bernoulli-Laplace Markov chain for $k>1$ was first studied by Nestoridi and  White \cite{NW}, who showed that when $k=o(n)$, $\frac{n}{4k} \log n - O\big(\frac{n}{k}\big)$ steps are necessary for the chain to approach stationarity, while $\frac{n}{2k} \log n$ steps are sufficient, in particular
\begin{equation} \label{eq:nestowhite}
\frac{n}{4k}\log n - \frac{c(\e)n}{k} \leq t_{\mr{mix}}(\e) \leq \frac{n}{2k} \log \Big(\frac{n}{\e}\Big).
\end{equation}
They also proved that if $\frac{n}{2} - \log_6 n\leq k
\leq  \frac{n}{2}$ then the chain mixes in a finite number of steps.

The main result of the present article provides the sharp upper bound for the mixing time of the $(n,k)$ Bernoulli--Laplace urn model when $k=o(n)$, thus bridging the gap \eqref{eq:nestowhite} from \cite{NW}.
\begin{theorem}\label{main}
The mixing time of the $(n,k)$ Bernoulli--Laplace urn model with $k=o(n)$ satisfies
\begin{equation} 
t_\mr{mix}(\e) \leq \frac{n}{4k}\log n + \frac{3n}{k}\log\log n + O\Big(\frac{n}{\e^4 k}\Big)
\end{equation}
for every $\e\in(0,1)$.
\end{theorem}
In combination with the sharp lower bound \eqref{eq:nestowhite} from \cite{NW}, we deduce that the $(n,k)$ Bernoulli--Laplace urn model with $k=o(n)$ exhibits cutoff at $\frac{n}{4k} \log n$ with window $\frac{n}{k}\log\log n$, i.e. that
\begin{equation*}
\lim_{c \rightarrow \infty} \lim_{n \rightarrow \infty} d \left(\frac{n}{4k} \log n -  \frac{cn}{k}\log\log n \right)= 1 \mbox{ and } \lim_{c \rightarrow \infty} \lim_{n \rightarrow \infty} d \left(\frac{n}{4k} \log n +  \frac{cn}{k} \log\log n \right)= 0.
\end{equation*}

The proof of Diaconis and Shahshahani for the upper bound when $k=1$, relies heavily on the spectrum of the transition matrix $P$. In particular, they use an upper bound for the total variation distance in terms of the $\ell_2$ norm of the eigenvalues and eigenvectors of the Markov chain (see also \cite[Section~12.6]{LPW}), which they then compute explicitly using spherical function theory. This spectral approach has proven successful in many variants of the $(n,1)$ Bernoulli--Laplace urn model, including the works \cite{DLS} and \cite{Eric} mentioned above. For instance, Scarabotti \cite{FS} studied the model where we have $m$ urns, each containing $n$ balls, and at each step we choose two balls that belong to different urns and switch them. He proved that this urn model exhibits cutoff at $\frac{1}{4}n (m-1) \log (nm^2)$. It is also worth mentioning that Schoolfield \cite{School} studied another version of the Bernoulli--Laplace Markov chain where at each step, a single ball is chosen from each urn and gets switched, but some balls get signed during the process.

Variants of the $(n,k)$ Bernoulli--Laplace urn model with $k>1$ have also been studied via spectral techniques. Most notably, Khare and Zhou \cite{KZ} (see also \cite{KM2}) analyzed the mixing time of a variety of multivariate urn models and proved cutoff for a simpler asymmetric version of the Bernoulli--Laplace chain, where at each step we pick $k$ balls from the left urn, move them to the right urn, and then we pick $k$ balls from the right urn and move them to the left. It turns out that the eigenvalues of this model are simple enough and thus an $\ell_2$ bound for the total variation distance \`a la Diaconis and Shahshahani yields the sharp upper bound for the mixing time. Moreover, Khare and Zhou found a closed formula for the eigenvalues and eigenvectors of the transition matrix $(P_x^t(y))_{x,y \in \mathcal{X}}$ of the $(n,k)$ Bernoulli--Laplace model that the present paper studies, yet, as those were significantly more complicated, they did not provide any estimates for its mixing time.

\subsection*{Outline of the proof} 
In contrast to the proofs of all the aforementioned results, throughout our proof we only make use of the first two eigenvalues and eigenvectors of our Markov chain. Beyond that, the proof is based on a coupling argument which is organized in 4 steps (each corresponding to a different subsection of Section \ref{sec:3}). Let $(X_t)$ and $(Y_t)$ be two copies of the Bernoulli-Laplace urn model with $k=o(n)$ swaps with $X_0\in\{0,1,\ldots,n\}$ being fixed and $Y_0\sim\pi_n$. Recall that $X_t$ and $Y_t$ denote the number of red balls in the left urn after $t$ steps at each configuration of the chain.

\smallskip

\noindent {\it Step 1.} Using the first two eigenvalues and eigenvectors of the Markov chain (see Subsections \ref{subsec:2.2} and \ref{subsec:2.3} below), we show that at time $t = \frac{n}{4k}\log n$ (which, recall, is the desired mixing time) both $X_t$ and $Y_t$ are of the order $\frac{n}{2} + O(\sqrt{n})$ with high probability. In particular, at this time, $X_t$ and $Y_t$ are with high probability at most $O(\sqrt{n})$ apart.

\smallskip

\noindent {\it Step 2.} Starting at time $\frac{n}{4k}\log n$ we run the two chains independently. A well-known hitting time lemma, (see Subsection \ref{subsec:2.4}) along with a combinatorial decomposition of the Markov chain, assures that with high probability the two configurations will be at distance $O(\sqrt{k\log n})$ away from each other within $O\big(\frac{n}{k}\big)$ steps. We note that to make the analysis possible, it is important to know that both chains will remain at distance, say, $\frac{n}{4}$ from $\frac{n}{2}$ for $O\big(\frac{n}{k}\big)$ additional steps, which is a consequence of Doob's maximal inequality (see Proposition \ref{prop:doob}).

\smallskip

\noindent {\it Step 3.} After the two chains get at distance $O(\sqrt{k\log n})$ from each other, we use a monotone coupling that was introduced in \cite{NW} (see Subsection \ref{subsec:2.1}) which guarantees that with high probability they will reach distance $o(\sqrt{k})$ within $\frac{3n}{k}\log\log n$ steps.

\smallskip

\noindent {\it Step 4.} Finally, we use a classical result of Diaconis and Freedman from \cite{DiaFree} to show that after the two configurations get to distance $o(\sqrt{k})$ from one another, then their total variation distance becomes $o(1)$ after a single step.

\smallskip

We note that when $k=n^{o(1)}$, Theorem \ref{main} can be proven along the lines of a proof of Levin, Luczak and Peres \cite{LLP}, who showed that the mixing time of Glauber dynamics for the mean-field Ising model  exhibits cutoff in the high temperature regime. This is explained in detail in Remarks \ref{rem:weakbounds} and \ref{rem:easycutoff} below. However, a straightforward adaptation of the technique of \cite{LLP} appears insufficient to prove the sharp upper bound for asymptotically larger values of $k=o(n)$ due to a number of technical complications, which are amended by the strategy outlined in Steps 1-4 above.

\smallskip

We conclude with a couple of natural open questions.

\begin{question} \label{q2}
What is the mixing time of the $(n,k)$ Bernoulli--Laplace model when $k/n \to \lambda\in\big(0,\frac{1}{2}\big)$?
\end{question}

Question \ref{q2} was studied in \cite{NW}, where it was shown that
\begin{equation} \label{eq:q2}
\frac{\log n}{2\log(1-2\lambda)^{-1}} - c(\e) \leq t_\mr{mix}(\e) \leq \frac{\log \big(\frac{n}{\e}\big)}{2\lambda(1-\lambda)}.
\end{equation}
We warn the reader that there exists a misprint in the statement of \cite[Theorem~4]{NW}, where the lower bound of the above result is wrongfully claimed to be $\frac{1}{4\lambda}\log n$, while in fact they prove the inequality that we present in \eqref{eq:q2}. We conjecture that the lower bound in \eqref{eq:q2} is sharp.

In view of the well known connection between mixing times and logarithmic Sobolev inequalities (see \cite{DSL96}), one is tempted to ask whether Theorem \ref{main} (and perhaps even an answer to Question \ref{q2}) can be recovered via the evaluation of the log-Sobolev constant of the chain.

\begin{question}
What is the log-Sobolev constant of the $(n,k)$ Bernoulli--Laplace urn model?
\end{question}

Finally, we wish to note that a version of the $(n,k)$ Bernoulli--Laplace urn model with multiple urns (generalizing the model of \cite{FS}) was studied in \cite{NW} and some estimates for its mixing time were obtained. As adapting the strategy of the present work to treat this model seems to require multiple technical (and perhaps even conceptual) modifications, we choose not to pursue it here; however we believe that the following question is of interest.

\begin{question} \label{q3}
What is the mixing time of the many-urn Bernoulli--Laplace model?
\end{question}

\subsection*{Asymptotic notation} In what follows we use the convention that for $a,b\in[0,\infty]$ the notation $a\gtrsim b$ or $b=O(a)$ (respectively $a\lesssim b$ or $a=O(b)$) means that there exists a universal constant $c\in(0,\infty)$ such that $a\geq cb$ (respectively $a\leq cb$). Moreover, $a\asymp b$ stands for $(a\lesssim b)\wedge(a\gtrsim b)$. Finally, we write $a=o(b)$ if $a$ is a function of $b$ with $\lim_{b\to\infty} a/b=0$.


\section{Preliminaries} \label{sec:2}

In this section, we will present various simple properties of the Bernoulli--Laplace urn model along with some classical estimates, which we shall later use for the proof of Theorem \ref{main}. Here and throughout, we will denote by $(X_t)$ the Bernoulli--Laplace chain on $\{0,1,\ldots,n\}$ with transition matrix
\begin{equation} \label{eq:transition}
P(i,j) = \begin{cases}
\mathlarger{\sum}\limits_{m=0}^{\min\{i,n-j,k-j+i \}} \frac{\binom{i}{m}\binom{n-i}{k-m}\binom{n-i}{j-i+m}\binom{i}{k-j+i-m}}{\binom{n}{k}^2}, & \mbox{if } i\leq j \\ 
\ \ \ \mathlarger{\sum}\limits_{m=i-j}^{\min\{i,n-j,k\}} \frac{\binom{i}{m}\binom{n-i}{k-m}\binom{n-i}{m-i+j}\binom{i}{k+i-j-m}}{\binom{n}{k}^2}, &\mbox{if } i> j
\end{cases}
\end{equation}
and starting point $X_0$ and by $(Y_t)$ another copy of the chain with $Y_0$ distributed according to the stationary measure $\pi_n$. The pair $(P,\pi_n)$ is a reversible Markov chain.

We also note that the Bernoulli--Laplace urn model posesses the following symmetry. If $(X_t)$ is the Bernoulli--Laplace chain with $k$ out of $n$ balls being swapped at each step and $(\widetilde{X}_t)$ is the chain where instead $n-k$ out of $n$ balls are being swapped at each step, then the mixing time of $(X_t)$ coincides with the mixing time of $(\widetilde{X}_t)$. This is a consequence of the bijection constructed as follows: when two sets of $k$ balls are swapped in $(X_t)$, then the complements of those sets are swapped in $(\widetilde{X}_t)$. In particular, from now on we will always assume that $k\leq\frac{n}{2}$.


\subsection{A combinatorial coupling} \label{subsec:2.1} Recall that if $\mu, \nu$ are probability measures on the spaces $(\Omega,\ms{F})$ and $(\Omega',\ms{F}')$ respectively, then a coupling $\ms{C}$ of $\mu$ and $\nu$ is a probability measure on $(\Omega\times\Omega',\ms{F}\otimes\ms{F}')$ whose marginals are $\mu$ and $\nu$, i.e. such that for $A\in\ms{F}$ and $B\in\ms{F}'$,
$$\ms{C}\big( \{(x,y): \ x\in A\}\big) = \mu(A) \ \ \mbox{and} \ \ \ms{C}\big(\{ (x,y): \ y\in B\}\big) = \nu(B).$$
The first ingredient needed for the proof of Theorem \ref{main} is a coupling of two instances of the Bernoulli--Laplace chain, constructed in \cite{NW}, which was used there to show the suboptimal upper bound
$$t_{\mr{mix}}(\e) \leq \frac{n}{2k} \log\Big(\frac{n}{\e}\Big)$$
for the mixing time of the Bernoulli--Laplace model with $k=o(n)$ swaps. According to this coupling, let $(X_t)$ and $(Y_t)$ be two instances of the Bernoulli--Laplace chain. Recall that $X_t$ and $Y_t$ is the number of red balls in the left urn in each of them after $t$ steps. For each time $t\geq0$, label the balls in both realizations of the chain as follows. First, label the balls in the left urns using the numbers $\{1,\ldots,n\}$ so that all red balls preceed all white balls in the ordering induced by this labeling. Then, do the same for the right urns using the numbers $\{n+1,\ldots,2n\}$. Now, choose $A$ and $B$ two subsets of $\{1,\ldots,n\}$ and $\{n+1,\ldots,2n\}$ respectively, each of cardinality $k$, and swap the balls indexed by $A$ in the left urns with those indexed by $B$ in the right. Finally, denote by $X_{t+1}$ and $Y_{t+1}$ the updated number of red balls in each of the left urns and remove the labels. One can easily check that, by the construction of this coupling, we always have
\begin{equation} \label{eq:couplingismonotone}
|X_{t+1}-Y_{t+1}| \leq |X_t-Y_t|.
\end{equation}
Furthermore, it was shown in \cite{NW} that if $|X_t-Y_t|=1$ for some $t\geq0$ and $(X_{t+1},Y_{t+1})$ is generated according to the coupling above, then
\begin{equation} \label{eq:couplingdecrease}
\mb{E}\big( |X_{t+1}-Y_{t+1}| \big| X_t, Y_t \big) = 1-\frac{2k(n-k)}{n^2}.
\end{equation}
Imitating the proof of the {\it path coupling theorem} of Bubley and Dyer (see \cite{BD} or \cite[Theorem~14.6]{LPW}), one gets the following proposition by iterating \eqref{eq:couplingdecrease} along suitable paths.

\begin{proposition} \label{prop:coupling}
Let $(X_t)$ and $(Y_t)$ be two instances of the Bernoulli--Laplace chain with $k$ swaps coupled as in the discussion above. For $r\in(0,\infty)$, let
\begin{equation} \label{eq:taur}
\tau_\mr{couple}(r) \eqdef \min \{t: \ |X_t-Y_t| \leq r \}.
\end{equation}
Then, for every $t\in(0,\infty)$,
\begin{equation} \label{eq:lemcoupling}
\mb{P}\big( \tau_\mr{couple}(r)>t \ \big| \ X_0, Y_0\big) \leq \Big(1-\frac{2k(n-k)}{n^2}\Big)^t \frac{|X_0-Y_0|}{r}.
\end{equation}
\end{proposition}


\subsection{Eigenvalues and eigenvectors} \label{subsec:2.2}
As mentioned in the introduction, the proof of Diaconis and Shahshahani \cite{Bernoulli-Laplace} for the cutoff behavior of the Bernoulli--Laplace urn model with $k=1$ ball being swaped at each step relied on a spectral bound for the total variation distance that was firstly used in \cite{RT} and a careful estimation of the eigenvalues of the walk. Let $(W_t)$ be a Markov chain on a state space $\Omega$. Recall that $\lambda\in[-1,1]$ is an eigenvalue of $(W_t)$ corresponding to the eigenvector $f:\Omega\to\R$, if 
\begin{equation} \label{eq:eigenvalue}
\mb{E}\big( f(W_{t+1})|W_t\big) = \lambda f(W_t),
\end{equation}
for every (equivalently, for some) $t\geq0$. Since $(X_t)$ is an irreducible, aperiodic and reversible chain on $\{0,1,\ldots,n\}$, it has $n+1$ real eigenvalues $1=\lambda_0$ and $\lambda_1,\ldots,\lambda_{n}\in(-1,1)$ which depend on both $n$ and $k$. An explicit description of the eigenvalues and eigenvectors of the Bernoulli--Laplace urn model for a general number of swaps at each step is known (see, e.g., \cite{Bernoulli-Laplace} and \cite{NW}). The eigenvectors are instances of the dual Hahn polynomials (see \cite{KZ}) and if $f_i(x)$ is the $i$-th eigenvector, then the corresponding eigenvalue is $f_i(k)$. 

As the exact expressions for these eigenvectors and eigenvalues are complicated, an efficient estimation, in the spirit of \cite{Bernoulli-Laplace}, of the spectral bound for the total variation distance seems intractable. Our analysis below relies solely on the precise values of two eigenvalues and eigenvectors of $(X_t)$, which correspond to the linear and quadratic dual Hahn polynomials respectively. In particular, using that the function $f_1(x)=1-\frac{2x}{n}$ is an eigenfunction with corresponding eigenvalue $\lambda_1=f_1(k)$ and that the function
\begin{equation} \label{eq:2ndeigen}
f_2(x)=1-\frac{2(2n-1)x}{n^2}+\frac{2(2n-1)x(x-1)}{n^2(n-1)}
\end{equation}
is an eigenfunction corresponding to the eigenvalue $\lambda_2=f_2(k)$, we derive the following identities.

\begin{lemma} \label{lem:expandvarident}
For every $t\geq0$, we have
\begin{equation} \label{eq:expectationident}
\mb{E}\big(X_t|X_0\big) = \frac{n}{2} - \Big(1-\frac{2k}{n}\Big)^t \Big(\frac{n}{2}-X_0\Big)
\end{equation}
and
\begin{equation} \label{eq:varianceident}
\mr{Var}\big(X_t|X_0\big) =  \frac{n^2}{4(2n-1)} + \frac{n^2(n-1)}{2(2n-1)}f_2(k)^t f_2(X_0) - \frac{n^2}{4}\Big(1-\frac{2k}{n}\Big)^{2t}\Big(1-\frac{2X_0}{n}\Big)^2
\end{equation}
\end{lemma}

\begin{proof}
Iterating the following eigenvalue property
\begin{equation} \label{eq:eigenv1}
\mb{E}\Big( 1-\frac{2X_{t}}{n}\Big| X_{t-1} \Big) = \Big(1-\frac{2k}{n}\Big)\cdot\Big(1-\frac{2X_{t-1}}{n}\Big),
\end{equation}
we get that
\begin{equation*}
\mb{E}\Big(1-\frac{2X_t}{n} \Big| X_0\Big) = \Big(1-\frac{2k}{n}\Big)^t\Big(1-\frac{2X_0}{n}\Big),
\end{equation*}
which is equivalent to \eqref{eq:expectationident}. For \eqref{eq:varianceident}, notice that
\begin{equation} \label{eq:f1f2}
f_1(x)^2 = \frac{1}{2n-1} + \frac{2n-2}{2n-1} f_2(x),
\end{equation}
therefore by similar reasoning as above, we have
\begin{equation*}
\mb{E}\Big( \Big(1-\frac{2X_t}{n}\Big)^2\Big| X_0\Big) = \frac{1}{2n-1} + \frac{2n-2}{2n-1} f_2(k)^t f_2(X_0).
\end{equation*}
Hence,
\begin{equation*}
\begin{split}
\mr{Var}\big(X_t|X_0\big) & = \frac{n^2}{4} \mr{Var}\Big(1-\frac{2X_t}{n}\Big|X_0\Big) \\
& = \frac{n^2}{4(2n-1)} + \frac{n^2(n-1)}{2(2n-1)}f_2(k)^t f_2(X_0) - \frac{n^2}{4}\Big(1-\frac{2k}{n}\Big)^{2t}\Big(1-\frac{2X_0}{n}\Big)^2,
\end{split}
\end{equation*}
which is precisely \eqref{eq:varianceident}.
\end{proof}


\subsection{Variance bounds} \label{subsec:2.3} The following estimates are a straightforward consequence of Lemma \ref{lem:expandvarident}. Recall that the stationary measure $\pi_n$ defined in \eqref{eq:pin} has mean $\frac{n}{2}$ and variance $\frac{n^2}{4(2n-1)} \asymp n$.

\begin{lemma} \label{lem:expandvar}
Let $t\geq\frac{n}{4k}\log n$. Then
\begin{equation} \label{eq:expandvar}
\Big|\mb{E}\big(X_{t}|X_0\big)-\frac{n}{2}\Big| \leq \sqrt{n} \ \ \ \mbox{and} \ \ \ \mathrm{Var}\big(X_t|X_0\big)\leq n.
\end{equation}
\end{lemma}

\begin{proof}
By Lemma \ref{lem:expandvarident}, we see that
\begin{equation*}
\begin{split}
\Big|\mb{E}\big(X_t|X_0\big) - \frac{n}{2}\Big| \leq \Big(1-\frac{2k}{n}\Big)^t n \leq e^{-\frac{\log n}{2}} n =\sqrt{n}.
\end{split}
\end{equation*}
Using \eqref{eq:f1f2} again, we can rewrite \eqref{eq:varianceident} as
\begin{equation} \label{eq:varproof1}
\mathrm{Var}\big(X_t|X_0\big) = \frac{n^2}{4(2n-1)} \Big( 1-\Big(1-\frac{2k}{n}\Big)^{2t} \Big) + \frac{n^2(n-1)}{2(2n-1)} f_2(X_0) \Big( f_2(k)^t-\Big(1-\frac{2k}{n}\Big)^{2t}\Big).
\end{equation}
Notice that $f_2(X_0)\geq f_2(n/2)=-\frac{1}{2(n-1)}$ and, by \eqref{eq:f1f2},
\begin{equation} \label{eq:varproof2}
0\leq \Big(1-\frac{2k}{n}\Big)^{2t} - f_2(k)^t =  \Big(\frac{1}{2n-1} + \frac{2n-2}{2n-1} f_2(k) \Big)^t - f_2(k)^t \leq1,
\end{equation}
since $f_2(k) \leq \max\{f_2(0),f_2(n)\}=1$. Therefore, we deduce that
$$\mathrm{Var}\big(X_t|X_0\big) \leq \frac{n^2}{2(2n-1)} \leq n,$$
which concludes the proof of \eqref{eq:expandvar}.
\end{proof}

It is an immediate consequence of Lemma \ref{lem:expandvar} in combination with Chebyshev's inequality that for $t\geq\frac{n}{4k}\log n$ and $r\in(0,\infty)$, we have
\begin{equation} \label{eq:concentration}
\mb{P}\Big(\big| X_t-\frac{n}{2}\big| > r\ \Big|\ X_0\Big) \leq \mb{P}\Big( \big| X_t - \mb{E}\big(X_t|X_0\big)\big| > r-\sqrt{n} \ \Big| \ X_0\Big) \leq \frac{n}{(r-\sqrt{n})^2}.
\end{equation}
In particular, for $\kappa_1\in(2,\infty)$, we deduce that
\begin{equation} \label{eq:concsqrt}
\mb{P}\Big( \big| X_t-\frac{n}{2}\big| > \kappa_1 \sqrt{n} \ \Big| \ X_0\Big) \lesssim \frac{1}{\kappa_1^2}.
\end{equation}
Using a martingale argument, we will show that for $r\gtrsim\sqrt{n}$, the chain $(X_t)$ will remain at distance at least $r$ from $\frac{n}{2}$ for at least a constant multiple of $\frac{n}{k}$ additional steps, with probability of the same order as in \eqref{eq:concentration}. This will be of central importance in the ensuing discussion.

\begin{proposition} \label{prop:doob}
Let $0\leq t_1\leq t_2$. If $k=o(n)$, then for every $r\in(0,\infty)$ we have
\begin{equation} \label{eq:doob}
\mathbb{P}\bigg( \bigcup_{t\in[t_1,t_2]} \Big\{ \Big|X_t-\frac{n}{2}\Big|> r \Big\}  \ \bigg| \ X_0 \bigg) \lesssim \frac{\exp\big(\frac{4k(t_2-t_1)}{n}\big)n^2}{r^2} \max\Big\{\frac{1}{n},  \min\big\{f_2(k)^{t_2}, f_2(X_0)\big\}\Big\}.
\end{equation}
In particular, if $\big| X_0-\frac{n}{2}\big| \leq A\sqrt{n}$, where $A\in(0,\infty)$, then for $t_0\in\N$ and $\gamma\in(0,\infty)$, we have
\begin{equation} \label{eq:doobgoodstart}
\mathbb{P}\bigg( \bigcup_{t\in[t_0,t_0+\gamma n/k]} \Big\{ \Big|X_t-\frac{n}{2}\Big|> r  \Big\} \ \bigg| \ X_0  \bigg) \lesssim \frac{A^2 e^{4\gamma}n}{r^2}
\end{equation}
and
\begin{equation} \label{eq:loglog}
\mathbb{P}\bigg( \bigcup_{t\in[t_0,t_0+ \frac{n}{k}\log\log n]} \Big\{ \Big|X_t-\frac{n}{2}\Big|> r  \Big\} \ \bigg| \ X_0  \bigg) \lesssim \frac{A^2 n(\log n)^4}{r^2}.
\end{equation}
\end{proposition}

\begin{proof}
Notice that for $t\geq0$,
$$\Big| X_t-\frac{n}{2}\Big| > r \ \ \  \Longleftrightarrow \ \ \  \Big| 1-\frac{2X_t}{n}\Big| > \frac{2r}{n}$$
and define
\begin{equation}
M_t \eqdef \frac{1-\frac{2X_t}{n}}{\big(1-\frac{2k}{n}\big)^t},
\end{equation}
which is a martingale because of \eqref{eq:eigenv1}. Then, we have
$$\bigg\{\sup_{t\in[t_1,t_2]}\Big| 1-\frac{2X_t}{n}\Big|>\frac{2r}{n}\bigg\} \subseteq \bigg\{ \sup_{t\in[t_1,t_2]} |M_t| > \frac{2r}{n\big(1-\frac{2k}{n}\big)^{t_1}}\bigg\}$$
and therefore
$$\mathbb{P}\bigg( \bigcup_{t\in[t_1,t_2]} \bigg\{ \Big|X_t-\frac{n}{2}\Big|>r \bigg\} \ \bigg| \ X_0  \bigg) \leq \mathbb{P}\bigg( \sup_{t\in [t_1,t_2]} |M_t| >\frac{2r}{n\big(1-\frac{2k}{n}\big)^{t_1}} \ \bigg| \ X_0\bigg).$$
By Doob's maximal inequality, we get
\begin{equation*}
\begin{split}
\mathbb{P}\bigg( \sup_{t\in [t_1,t_2]} |M_t| & >\frac{2r}{n\big(1-\frac{2k}{n}\big)^{t_1}}\ \bigg| \ X_0\bigg)  \leq  \frac{n^2}{r^2}\big(1-\frac{2k}{n}\big)^{2t_1} \mb{E}\big(|M_{t_2}|^2\big| X_0\big) \\ & = \frac{n^2}{r^2} \Big(1-\frac{2k}{n}\Big)^{-2(t_2-t_1)} \mb{E}\big( f_1(X_{t_2})^2\big| X_0\big) \lesssim \frac{\exp\big(\frac{4k(t_2-t_1)}{n}\big)n^2}{r^2} \mb{E}\big( f_1(X_{t_2})^2\big| X_0\big).
\end{split}
\end{equation*}
Moreover, \eqref{eq:f1f2} gives the estimate
\begin{equation*}
\begin{split}
\mb{E}\big(f_1(X_{t_2})^2\big| X_0\big) & = \frac{1}{2n-1} +\frac{2n-2}{2n-1} f_2(k)^{t_2} f_2(X_0) \lesssim \max\Big\{ \frac{1}{n}, f_2(k)^{t_2} f_2(X_0) \Big\},
\end{split}
\end{equation*}
and \eqref{eq:doob} follows since $f_2(x) \leq 1$ for every $x\in\{0,1,\ldots,n\}$. If additionally, $\big|X_0-\frac{n}{2}\big| \leq A\sqrt{n}$,
$$f_2(X_0)\leq f_2\Big(\frac{n}{2}+A\sqrt{n}\Big) \stackrel{\eqref{eq:f1f2}}{\leq} f_1\Big(\frac{n}{2}+A\sqrt{n}\Big)^2 \lesssim \frac{A^2}{n},$$
therefore, \eqref{eq:doob} implies that
$$\mathbb{P}\bigg( \bigcup_{t\in[t_0,t_0+\gamma n/k]} \Big\{ \Big|X_t-\frac{n}{2}\Big|> r  \Big\} \ \bigg| \ X_0  \bigg) \lesssim \frac{e^{4\gamma}n^2}{r^2} \cdot \frac{A^2}{n} = \frac{A^2 e^{4\gamma}n}{r^2},$$
which is exactly \eqref{eq:doobgoodstart}. Similarly,
$$\mathbb{P}\bigg( \bigcup_{t\in[t_0,t_0+ \frac{n}{k}\log\log n]} \Big\{ \Big|X_t-\frac{n}{2}\Big|> r  \Big\} \ \bigg| \ X_0  \bigg) \lesssim \frac{e^{4\log\log n} n^2}{r^2}\cdot \frac{A^2}{n} =  \frac{A^2 n(\log n)^4}{r^2},$$
which completes the proof.
\end{proof}

For $\gamma\in[0,\infty)$, let $t_{n,k}(\gamma)\eqdef \frac{n}{4k}\log n + \frac{\gamma n}{k}$. Since 
$$f_2(k)^{t_{n,k}(\gamma)} \leq f_1(k)^{2t_{n,k}(\gamma)}= \Big(1-\frac{2k}{n}\Big)^{\frac{n}{2k}\log n + \frac{2\gamma n}{k}} \asymp \frac{1}{ne^{4\gamma}},$$
it follows from equation \eqref{eq:doob} that for every $\kappa, \gamma\in(0,\infty)$, we have
\begin{equation} \label{eq:doobsqrt}
\mb{P}\bigg(\bigcup_{t\in[t_{n,k}(0),t_{n,k}(\gamma)]} \Big\{ \Big| X_t-\frac{n}{2}\Big| > \kappa \sqrt{n}  \Big\} \ \bigg| \ X_0 \bigg) \lesssim \frac{e^{4\gamma}}{\kappa^2},
\end{equation}
which improves upon \eqref{eq:concsqrt} up to the universal constants in the right hand side. For fixed $\kappa,\gamma\in(0,\infty)$, consider the event
\begin{equation} \label{eq:Eevent}
E_{\kappa,\gamma} = \bigg\{ \Big|X_t-\frac{n}{2}\Big| \leq \kappa \sqrt{n}, \ \mbox{for every } t\in[t_{n,k}(0),t_{n,k}(\gamma)]\bigg\}.
\end{equation}
Then \eqref{eq:doobsqrt} can be rewritten as $\mb{P}(E_{\kappa,\gamma}^\msf{c})\lesssim e^{4\gamma}/\kappa^2$, where $E_{\kappa,\gamma}^\msf{c}$ is the complement of $E_{\kappa,\gamma}$.


\subsection{Hitting time estimates} \label{subsec:2.4}
Throughout the proof, we will make use of the following well-known hitting time lemma appearing, e.g., in \cite[Proposition~17.20]{LPW}.

\begin{lemma} \label{lem:lpw}
Let $(Z_t)$ be a non-negative supermartingale adapted to the filtration $(\ms{F}_t)$ and $\tau$ be  a stopping time. Suppose that
\begin{enumerate} [$\bullet$]
\item $Z_0 = z_0$,
\item $|Z_{t+1}-Z_t|\leq B$, for every $t\geq0$ and
\item there exists $\sigma\in(0,\infty)$ such that $\mr{Var}(Z_{t+1} | \ms{F}_t) > \sigma^2$ on the event $\{\tau>t\}$.
\end{enumerate}
Then, for every $u>12B^2/\sigma^2$, we have
\begin{equation} \label{eq:lpw}
\mb{P} (\tau>u) \leq \frac{4z_0}{\sigma \sqrt{u}}.
\end{equation}
\end{lemma}

A straightforward application of Lemma \ref{lem:lpw} combined with Proposition \ref{prop:doob} and Proposition \ref{prop:coupling} is sufficient to prove cutoff for the Bernoulli--Laplace chain with $k=n^{o(1)}$ balls being swapped at each step. This is carried out in detail in Remark \ref{rem:weakbounds} and Remark \ref{rem:easycutoff}. To prove cutoff for asymptotically larger values of $k$, one needs to take into account that even though $\|X_{s+1}-X_s\|_\infty= 2k$ for every $s\geq0$, if $t\geq\frac{n}{4k}\log n$, then
\begin{equation} \label{eq:Xtdrift}
\mb{E}\big( |X_{t+1}-X_t| \ \big| \ X_0\big) \lesssim \sqrt{k}.
\end{equation}
This can be proven by first applying Jensen's inequality
$$\mb{E}\big( |X_{t+1}-X_t| \ \big| \ X_0\big) \leq \sqrt{\mb{E}\big( |X_{t+1}-X_t|^2 \ \big| \ X_0\big)}$$
and then employing the first two eigenvalues of the chain, as in the proof of Lemma \ref{lem:expandvarident}. We omit the details of this computation, since \eqref{eq:Xtdrift} will be a consequence of a more tiresome (but necessary for the proof) technical statement below (see Lemma \ref{lem:hoeffding}).

\section{Proof of Theorem \ref{main}} \label{sec:3}

We proceed with the proof of the main result of this article. The argument is divided in 4 subsections, each reflecting one of the steps explained in the outline presented in the introduction.


\subsection{Getting at distance $O\big( \sqrt{n} \big)$} \label{subsec:3.1}
Let $(X_t)$ and $(Y_t)$ be two instances of the Bernoulli--Laplace chain with $k=o(n)$ swaps at each step, where $X_0\in\{0,1,\ldots,n\}$ and $Y_0$ is distributed according to the stationary measure $\pi_n$. We will first show that with high probability, at time $\frac{n}{4k}\log n$, the difference $|X_t-Y_t|$ is smaller than a constant multiple of $\sqrt{n}$. The variance estimate \eqref{eq:concsqrt} of Subsection \ref{subsec:2.3} along with the union bound immediately imply the following lemma.

\begin{lemma} \label{lem:usevariance}
Let $(X_t)$, $(Y_t)$ be two copies of the Bernoulli--Laplace chain with $k$ swaps. For $\kappa_1\in(0,\infty)$, let
\begin{equation} \label{eq:deftau1}
\tau_1(\kappa_1) \eqdef \min\Big\{t: \ X_t, Y_t \in \Big(\frac{n}{2}-\kappa_1\sqrt{n}, \frac{n}{2}+\kappa_1\sqrt{n}\Big)\Big\}.
\end{equation}
Then,
\begin{equation} \label{eq:boundtau1}
\mb{P}\Big( \tau_1(\kappa_1) > \frac{n}{4k}\log n \ \Big| \ X_0, Y_0\Big) \lesssim \frac{1}{\kappa_1^2}.
\end{equation}
\end{lemma}

\begin{proof}
Let $t_0 = \frac{n}{4k}\log n$. Then
\begin{equation*}
\begin{split}
\mb{P}\Big( \tau_1(\kappa_1) > \frac{n}{4k}\log n\Big| X_0, Y_0\Big) \leq \mb{P}\Big( \big| X_{t_0}-\frac{n}{2}\big| > \kappa_1\sqrt{n} \Big| X_0\Big) + \mb{P}\Big( \big| Y_{t_0}-\frac{n}{2}\big| & >\kappa_1\sqrt{n}\Big| Y_0\Big) \stackrel{\eqref{eq:concsqrt}}{\lesssim} \frac{1}{\kappa_1^2},
\end{split}
\end{equation*}
which is precisely \eqref{eq:boundtau1}.
\end{proof}


\subsection{Getting at distance $O\big(\sqrt{k\log n}\big)$} \label{subsec:3.2}

Let $(X_t)$ and $(Y_t)$ be two instances of the Bernoulli--Laplace chain as before. We will now argue that if $(X_t)$ is independent of $(Y_t)$ after the stopping time $\tau_1(\kappa_1)$ of Lemma \ref{lem:usevariance}, then with high probability the difference $|X_t-Y_t|$ will become smaller than $O(\sqrt{k\log n})$ within $O\big(\frac{n}{k}\big)$ steps. 

Recall that in the combinatorial description of the chain, we denote by $X_t$ the number of red balls in the left urn after $t$ steps. For each $s\geq0$, we will define a new pair of (time inhomogeneous) Markov chains $(\msf{X}_s^{\mr{left}})$ and $(\msf{X}_s^{\mr{right}})$ on $\{0,1,\ldots,n\}$ with the property that for every $t\geq0$,
\begin{equation} \label{eq:distributioneq}
\msf{X}_{2kt}^{\mr{left}} \stackrel{d}{=} X_t \ \ \ \mbox{and} \ \ \ \msf{X}_{2kt}^{\mr{right}} \stackrel{d}{=} n-X_t,
\end{equation}
where $\stackrel{d}{=}$ denotes equality in distribution. For time $s=0$, denote by $\msf{X}_{0}^{\mr{left}}=X_0$ and $\msf{X}_{0}^{\mr{right}} =n-X_0$. For each time $s\in\{1,\ldots,k\}$ pick a ball uniformly at random from the left urn and place it in a storage space by the right urn. We will refer to the balls placed in the storage space as {\it unavailable} balls and to the rest as {\it available} balls. Notice that for every $s\in\{1,\ldots,k\}$, after the $s$-th ball has been moved from left to right, there are $n-s$ balls in the left urn and $n+s$ balls in the right urn, $s$ of which are unavailable. For each time $s\in\{k+1,\ldots,2k\}$ pick an available ball uniformly at random from the right urn and place it in the left urn. Finally, after $k$ balls have been moved back to the left urn, label all unavailable balls as available, place them in the right urn and restart the process. After $s$ steps of this process, denote by $\msf{X}_s^{\mr{left}}$ the number of red balls in the left urn and by $\msf{X}_s^{\mr{right}}$ the number of {\it available} red balls in the right urn. It is evident from the construction that
\begin{equation*}
\msf{X}_{2k}^{\mr{left}} \stackrel{d}{=} X_1 \ \ \ \mbox{and} \ \ \ \msf{X}_{2k}^{\mr{right}} \stackrel{d}{=} n-X_1.
\end{equation*}
To define $\msf{X}_s^{\mr{left}}$ and $\msf{X}_s^{\mr{right}}$ for general $s\geq0$, repeat the above process periodically after $2k$ steps; then \eqref{eq:distributioneq} clearly holds for any $t\geq0$. Notice that for every $s\geq0$, this construction gives
\begin{equation} \label{eq:decincrements}
\msf{X}^\mr{left}_{s+1}-\msf{X}^\mr{left}_s \in \{-1,0,1\},
\end{equation}
yet the same does not hold for $(\msf{X}_s^\mr{right})$. Similarly, one can define the decomposed chains $(\msf{Y}_s^{\mr{left}}), (\msf{Y}_s^{\mr{right}})$ corresponding to $(Y_t)$. The following is a consequence of the hitting time estimate of Lemma \ref{lem:lpw}.

\begin{lemma} \label{lem:hitting}
Let $(X_t), (Y_t)$  be two independent copies of the Bernoulli--Laplace chain with $k=o(n)$ swaps and let $(\msf{X}_s^{\mr{left}}), (\msf{X}_s^{\mr{right}})$ and $(\msf{Y}_s^{\mr{left}}), (\msf{Y}_s^{\mr{right}})$ be the decomposed chains constructed above. Assume that $X_0, Y_0 \in\big(\frac{n}{2}-\kappa_1\sqrt{n}, \frac{n}{2}+\kappa_1\sqrt{n}\big)$ for some $\kappa_1\in(0,\infty)$. Consider the stopping time
\begin{equation} \label{eq:deftau}
\tau_\mr{match} \eqdef \min\{s: \ \msf{X}_{s}^{\mathrm{left}} = \msf{Y}_{s}^{\mathrm{left}} \ \ \mbox{or} \ \ \msf{X}_{s}^{\mr{right}} = \msf{Y}_{s}^{\mr{right}}\}.
\end{equation}
Then, for every $\gamma_1\in(0,\infty)$ and large enough values of $n$, we have
\begin{equation} \label{eq:lemtaumatch}
\mb{P}(\tau_\mr{match} > \gamma_1 n\ | \ X_0, Y_0) \lesssim \frac{\kappa_1}{\sqrt{\gamma_1}}.
\end{equation}
\end{lemma}

\begin{proof}
For $s\geq0$, let $\ms{F}_s=\sigma\big( \msf{X}_r^{\mr{left}}, \msf{X}_r^{\mr{right}}, \msf{Y}_r^{\mr{left}}, \msf{Y}_r^{\mr{right}} \big)_{r\leq s}$ be the $\sigma$-algebra generated by the path of the two decomposed chains up to time $s$. For Lemma \ref{lem:lpw} to be applied, it is important that the decomposed chains remain close to $\frac{n}{2}$. To this end, consider the truncated stopping time
\begin{equation} \label{eq:taukappa3}
\begin{split}
\widetilde{\tau}_\mr{match}\eqdef\tau_\mr{match}\wedge\min\Big\{2ks: \max_{r\in[0,s]}\Big| X_r - \frac{n}{2} \Big|\vee \Big| Y_r-\frac{n}{2}\Big| > \frac{n}{4}\Big\}.
\end{split}
\end{equation}
By equation \eqref{eq:doobgoodstart} of Proposition \ref{prop:doob}, we have
\begin{equation} \label{eq:firstdecomp}
\begin{split}
\mb{P}(\tau_\mr{match}>\gamma_1 n | X_0, Y_0) = \mb{P} ( \widetilde{\tau}_\mr{match} > \gamma_1 n & | X_0, Y_0)  + \mb{P}( \widetilde{\tau}_\mr{match} \leq \gamma_1 n  < \tau_\mr{match} | X_0, Y_0) \\ & \lesssim \mb{P} ( \widetilde{\tau}_\mr{match} >\gamma_1 n \ | \ X_0, Y_0 ) + \frac{\kappa_1^2 e^{2\gamma_1}}{n},
\end{split}
\end{equation}
since
$$\widetilde{\tau}_{\mathrm{match}} \leq \gamma_1 n < \tau_\mr{match} \ \ \Longrightarrow \ \ \max_{t\in[0, \gamma_1 n/2k]} \Big|X_t - \frac{n}{2}\Big| \vee \Big| Y_t-\frac{n}{2}\Big| > \frac{n}{4}.$$
and $X_0, Y_0\in\big(\frac{n}{2}-\kappa_1\sqrt{n},\frac{n}{2}+\kappa_1\sqrt{n}\big)$. Consider the stochastic process $W_s = \msf{X}_{s}^{\mr{left}} - \msf{Y}_{s}^{\mr{left}}$ and let $Z_s = W_{s\wedge \tau_\mr{match}}$. It follows from \eqref{eq:decincrements} and \eqref{eq:deftau} that $Z_s\geq0$ for every $s\geq0$. Furthermore, $(Z_s)$ is a super-martingale. To see this, notice that for $s<\tau_\mr{match}$, we have $\msf{X}_{s}^\mr{left}>\msf{Y}_{s}^{\mr{left}}$ and $\msf{X}_{s}^\mr{right} <\msf{Y}_{s}^\mr{right}$. Therefore, if $r\eqdef s\ \!(\!\!\!\!\mod2k) \in\{0,1,\ldots,k-1\}$, then
$$\mb{E}\big(\msf{X}_{s+1}^\mr{left}-Y_{s+1}^{\mr{left}}\ \big| \ \ms{F}_s\big) - \big(\msf{X}_s^\mr{left}-\msf{Y}_s^\mr{left}\big) = \frac{-\msf{X}_s^\mr{left}}{n-r} + \frac{\msf{Y}_s^\mr{left}}{n-r}=\frac{\msf{Y}_s^\mr{left}-\msf{X}_s^\mr{left}}{n-r} < 0$$
and if $r \in\{k,k+1,\ldots,2k-1\}$, then
$$\mb{E}\big(\msf{X}_{s+1}^\mr{left}-Y_{s+1}^{\mr{left}}\ \big| \ \ms{F}_s\big) - \big(\msf{X}_s^\mr{left}-\msf{Y}_s^\mr{left}\big) = \frac{\msf{X}_s^\mr{right}}{n-r+k}- \frac{\msf{Y}_s^\mr{right}}{n-r+k} =  \frac{\msf{X}_s^\mr{right}-\msf{Y}_s^\mr{right}}{n-r+k} < 0.$$
A similar calculation, shows that for $s<\widetilde{\tau}_\mr{match}$ and $r\in\{0,1,\ldots,k-1\}$,
\begin{equation*}
\begin{split}
\mr{Var}(Z_{s+1}|\ms{F}_s) = \mr{Var}(\msf{X}^\mr{left}_{s+1}|\ms{F}_s)+\mr{Var}(\msf{Y}^\mr{left}_{s+1}|\ms{F}_s) & = \frac{\msf{X}_s^\mr{left} ( n-r-\msf{X}_s^\mr{left})}{(n-r)^2} + \frac{\msf{Y}_s^\mr{left} ( n-r-\msf{Y}_s^\mr{left})}{(n-r)^2} \\ & \geq \frac{2\big(\frac{n}{4}-k\big)\big(\frac{n}{4}-r-k\big)}{n^2} \gtrsim 1,
\end{split}
\end{equation*}
for $n$ large enough, where the first identity follows from the independence of $(\msf{X}_s^\mr{left})$ and $(\msf{Y}_s^\mr{left})$ and the second to last inequality from the definition of $\widetilde{\tau}_\mr{match}$. The same holds true if $r\in\{k,k+1,\ldots,2k-1\}$. Therefore, applying Lemma \ref{lem:lpw} for the stopping time $\widetilde{\tau}_\mr{match}$, we get that
\begin{equation} \label{eq:boundtauk1}
\mb{P}(\widetilde{\tau}_\mr{match}>\gamma_1 n \ | \ X_0, Y_0) \lesssim \frac{\kappa_1\sqrt{n}}{\sqrt{\gamma_1 n}} = \frac{\kappa_1}{\sqrt{\gamma_1}}.
\end{equation}
Combining \eqref{eq:firstdecomp} and \eqref{eq:boundtauk1}, we finally deduce that
$$\mb{P}(\tau_\mr{match}>\gamma_1 n \ | \ X_0, Y_0) \lesssim \frac{\kappa_1}{\sqrt{\gamma_1}} + \frac{\kappa_1^2 e^{2\gamma_1}}{n} \lesssim \frac{\kappa_1}{\sqrt{\gamma_1}}$$
for large enough values of $n$.
\end{proof}

We will also need the following technical lemma, which is in the spirit of Proposition \ref{prop:doob}.

\begin{lemma} \label{lem:hoeffding}
Let $(X_t)$ be the Bernoulli--Laplace chain with $k=o(n)$ swaps and starting point $X_0\in\big(\frac{n}{2}-\kappa_1\sqrt{n},\frac{n}{2}+\kappa_1\sqrt{n}\big)$ for some $\kappa_1\in(0,\infty)$. Then, for every $\gamma_1\in(0,\infty)$ and $\kappa_2\in(0,\infty)$, we have
\begin{equation} \label{eq:hoeffding}
\mb{P}\bigg(\bigcup_{\substack{t\in[0,\gamma_1 n/2k] \\ r\in\{1,\ldots,k\}}} \Big\{ \Big| \msf{X}_{2kt}^\mr{left}-\msf{X}_{2kt+r}^\mr{left}-\frac{r}{2} \Big| \vee \Big| \msf{X}_{2kt}^\mr{right}-\msf{X}^\mr{right}_{(2t+1)k+r}-\frac{r}{2} \Big| > \kappa_2\sqrt{k\log n} \Big\} \ \bigg| \ X_0 \bigg)\lesssim \frac{\kappa_1^2e^{2\gamma_1}}{\kappa_2}
\end{equation}
and
\begin{equation} \label{eq:hoeffding2}
\mb{P} \bigg( \bigcup_{t\in[0,\gamma_1 n/2k]} \bigcup_{r\in\{1,\ldots,k\}} \Big\{ \Big| \msf{X}^\mr{left}_{2kt}-\msf{X}_{(2t+1)k+r}^\mr{left} - \frac{k-r}{2} \Big| > \kappa_2\sqrt{k\log n} \Big\} \ \bigg| \ X_0 \bigg)\lesssim \frac{\kappa_1^2e^{2\gamma_1}}{\kappa_2}.
\end{equation}
\end{lemma}

\begin{proof}
We can clearly assume that $\kappa_2\geq 100\kappa_1^2e^{2\gamma_1}$, since otherwise the conclusion is trivial. We will first prove \eqref{eq:hoeffding}. Consider the event
\begin{equation} \label{eq:theeventE}
E_{\kappa_2,\gamma_1} \eqdef \bigg\{ \Big|X_t-\frac{n}{2}\Big| \leq \kappa_2 \sqrt{n}, \ \mbox{for every } t\in\big[0,\gamma_1 n/2k\big]\bigg\}
\end{equation}
and notice that equation \eqref{eq:doobgoodstart} of Proposition \ref{prop:doob} implies that $\mb{P}(E_{\kappa_2,\gamma_1}^\msf{c})\lesssim \kappa_1^2 e^{2\gamma_1}/\kappa_2^2$. Hence, if
\begin{equation} \label{eq:Fevent}
F_{\kappa_2,\gamma_1} \eqdef \bigcap_{t\in[0,\gamma_1 n/2k]} \bigcap_{r\in\{1,\ldots,k\}} \Big\{ \Big| \msf{X}^\mr{left}_{2kt} - \msf{X}_{2kt+r}^\mr{left} -\frac{r}{2} \Big| \vee \Big| \msf{X}_{2kt}^\mr{right}-\msf{X}^\mr{right}_{(2t+1)k+r}-\frac{r}{2} \Big| \leq \kappa_2\sqrt{k\log n}\Big\},
\end{equation}
then
\begin{equation} \label{eq:breakF}
\mb{P}\big(F^\msf{c}_{\kappa_2,\gamma_1}\ \big| \ X_0\big) \lesssim \mb{P}\big(F^\msf{c}_{\kappa_2,\gamma_1}\big| E_{\kappa_2,\gamma_1}\big) + \frac{\kappa_1^2 e^{2\gamma_1}}{\kappa_2^2}.
\end{equation}
We will control the probability involving the terms where $\msf{X}^\mr{left}$ appears, and the ones involving $\msf{X}^\mr{right}$ can be bounded similarly. To this end, consider the random variable
$$Z\eqdef\max_{t\in[0,\gamma_1 n/2k]} \max_{r\in\{1,\ldots,k\}} \Big|\msf{X}_{2kt}^\mr{left}-\msf{X}_{2kt+r}^\mr{left}-\frac{r}{2}\Big|$$
and notice that for every $t\in\N$ and $r\in\{1,\ldots,k\}$, the difference $\msf{X}_{2kt}^\mr{left} - \msf{X}_{2kt+r}^\mr{left}$ conditioned on the value of $\msf{X}_{2kt}^\mr{left}$ is a hypergeometric random variable with parameters $(n,\msf{X}_{2kt}^\mr{left},r)$, i.e. 
\begin{equation} \label{eq:hypergeoHoeffding}
\mb{P}\big( \msf{X}_{2kt}^\mr{left}-\msf{X}_{2kt+r}^\mr{left} = j \ \big| \ \msf{X}_{2kt}^\mr{left}\big) = \frac{\binom{\msf{X}_{2kt}^\mr{left}}{j}\binom{n-\msf{X}_{2kt}^\mr{left}}{r-j}}{\binom{n}{r}}, \ \ \ \mbox{where } j\in\{0,1,\ldots,r\}.
\end{equation}
Therefore, on $E_{\kappa_2, \gamma_1}$ we have
$$\Big| \mb{E}\big(  \msf{X}_{2kt}^\mr{left}-\msf{X}_{2kt+r}^\mr{left} \ \big| \ \msf{X}_{2kt}^\mr{left}\big)-\frac{r}{2}\Big| = \Big| \frac{r\msf{X}_{2kt}^\mr{left}}{n} - \frac{r}{2}\Big|\leq \frac{k}{n} \Big| \msf{X}_{2kt}^\mr{left} - \frac{n}{2}\Big| \leq \frac{\kappa_2 k}{\sqrt{n}} = o\big(\sqrt{k}\big)$$
and thus, by the triangle inequality,
\begin{equation} \label{eq:hoeffdingusexp}
\mb{P}\Big( Z> \kappa_2\sqrt{k \log n} \ \Big| \  E_{\kappa_2,\gamma_1} \Big) \leq \mb{P}\Big( Z'> \frac{1}{2}\kappa_2\sqrt{k \log n} \ \Big| \ E_{\kappa_2,\gamma_1} \Big),
\end{equation}
where
$$Z' \eqdef \max_{t\in[0,\gamma_1n/2k]} \max_{r\in\{1,\ldots,k\}} \big| \msf{X}^\mr{left}_{2kt}-\msf{X}^\mr{left}_{2kt+r}-\mb{E}\big( \msf{X}_{2kt}^\mr{left}-\msf{X}_{2kt+r}^\mr{left} \big| \msf{X}_{2kt}^\mr{left}\big)\big|.$$
By symmetry of the event $E_{\kappa_2,\gamma_1}$, it suffices to bound
$$\mb{P}\Big( Z''> \frac{1}{4}\kappa_2\sqrt{k \log n}  \ \Big| \ E_{\kappa_2,\gamma_1} \Big),$$
where
$$Z'' \eqdef \max_{t\in[0,\gamma_1n/2k]} \max_{r\in\{1,\ldots,k\}} \msf{X}^\mr{left}_{2kt}-\msf{X}^\mr{left}_{2kt+r}-\mb{E}\big( \msf{X}_{2kt}^\mr{left}-\msf{X}_{2kt+r}^\mr{left} \big| \msf{X}_{2kt}^\mr{left}\big).$$
Let $H_{t,r} \eqdef \msf{X}^\mr{left}_{2kt}-\msf{X}^\mr{left}_{2kt+r}$. Then, for every $h\in(0,\infty)$, Jensen's inequality gives
\begin{equation} \label{eq:hoeffdingjensen}
\begin{split}
\exp\big(h \mb{E}(Z'' | E_{\kappa_2,\gamma_1})\big) & \leq \mb{E}\big(\exp(hZ'') \big| E_{\kappa_2,\gamma_1}\big) \\ &\leq \sum_{t=0}^{\lfloor \gamma_1 n/2k\rfloor} \sum_{r=1}^k \mb{E} \big( \exp \big( h\big(H_{t,r} - \mb{E}(H_{t,r}|\msf{X}_{2kt}^\mr{left}))\big) \big| E_{\kappa_2,\gamma_1}\big) \\ & =\sum_{t=0}^{\lfloor \gamma_1n/2k\rfloor}\sum_{r=1}^k \mb{E}\big( \mb{E}\big( \exp \big( h\big(H_{t,r} - \mb{E}(H_{t,r}|\msf{X}_{2kt}^\mr{left}))\big) \big| E_{\kappa_2,\gamma_1}, \msf{X}_{2kt}^\mr{left} \big))\big) ,
\end{split}
\end{equation}
where the last equality is the tower property of the conditional expectation. Fix $t\in\{0,\ldots,\lfloor \gamma_1 n/2k\rfloor\}$ and $r\in\{1,\ldots,k\}$.  On $E_{\kappa_2, \gamma_1}$, we have $\msf{X}_{2kt}^\mr{left} \in \big(\frac{n}{2}-\kappa_2\sqrt{n}, \frac{n}{2}+\kappa_2\sqrt{n}\big)$ and, moreover, $\mb{P}(E_{\kappa_2,\gamma_1})\gtrsim 1$, since $\kappa_2\geq100\kappa_1^2e^{2\gamma_1}$. Consequently,
\begin{equation} \label{eq:reduceconditioning}
\begin{split}
\mb{E}\big( \exp \big( h\big(H_{t,r} - & \mb{E}(H_{t,r} | \msf{X}_{2kt}^\mr{left}))\big) \ \big| \ E_{\kappa_2,\gamma_1}, \msf{X}_{2kt}^\mr{left} \big)\\ & \lesssim \mb{E}\Big( \exp \big( h\big(H_{t,r} - \mb{E}(H_{t,r}|\msf{X}_{2kt}^\mr{left}))\big)\  \Big| \ \msf{X}_{2kt}^\mr{left}\in \Big(\frac{n}{2}-\kappa_2\sqrt{n}, \frac{n}{2}+\kappa_2\sqrt{n}\Big) \Big)
\end{split}
\end{equation}
Recall that $H_{t,r} = \msf{X}^\mr{left}_{2kt}-\msf{X}^\mr{left}_{2kt+r}$ is distributed according to the hypergeometric distribution with parameters $(n, \msf{X}_{2kt}^\mr{left}, r)$ when conditioned on the value of $\msf{X}_{2kt}^\mr{left}$. Therefore, it is well known (see \cite[Section~6]{hoeff} or \cite[Theorem~2.2]{serf}) that since $k=o(n)$,
\begin{equation} \label{eq:hoeffdingsbound}
\begin{split}
\mb{E}\Big( \exp \big( h\big(H_{t,r} - \mb{E}(H_{t,r}|\msf{X}_{2kt}^\mr{left}))\big)\  \Big| \ \msf{X}_{2kt}^\mr{left}\in \Big(\frac{n}{2}-&\kappa_2\sqrt{n}, \frac{n}{2}+\kappa_2\sqrt{n}\Big) \Big) \\ & \leq \exp\big( h^2 r/16 \big) \leq \exp\big( h^2k/16\big),
\end{split}
\end{equation}
which, combined with \eqref{eq:hoeffdingjensen} and \eqref{eq:reduceconditioning} gives
$$\exp\big(h \mb{E}(Z'' | E_{\kappa_2,\gamma_1})\big) \lesssim \gamma_1 n \exp\big(h^2k/16\big).$$
Taking logarithms and dividing by $h$, we deduce that
\begin{equation} \label{eq:hoeffdingexpmax}
\mb{E}\big(Z''\ \big| \ E_{\kappa_2,\gamma_1}\big) \lesssim \frac{\log n}{h} + hk \asymp \sqrt{k\log n},
\end{equation}
for $h\asymp \sqrt{\frac{\log n}{k}}$. Finally, by Markov's inequality, we get
$$\mb{P}\Big( Z''> \frac{1}{4}\kappa_2\sqrt{k \log n}  \ \Big| \ E_{\kappa_2,\gamma_1} \Big) \lesssim \frac{\mb{E}\big(Z''|E_{\kappa_2,\gamma_1}\big)}{\kappa_2\sqrt{k\log n}} \stackrel{\eqref{eq:hoeffdingexpmax}}{\lesssim} \frac{1}{\kappa_2}. $$
As explained earlier, this implies that
$$\mb{P}\big(F^\msf{c}_{\kappa_2,\gamma_1}\big| E_{\kappa_2,\gamma_1}\big) \lesssim \frac{1}{\kappa_2}$$
which combined with \eqref{eq:breakF} completes the proof of \eqref{eq:hoeffding}. To deduce \eqref{eq:hoeffding2} from \eqref{eq:hoeffding}, notice that
$$\msf{X}_{2kt}^\mr{left} - \msf{X}_{(2t+1)k+r}^\mr{left} = \big(\msf{X}_{2kt}^\mr{left} - \msf{X}_{(2t+1)k}^\mr{left}\big) - \big(\msf{X}_{2kt}^\mr{right} - \msf{X}_{(2t+1)k+r}^\mr{right}\big),$$
hence
\begin{equation*}
\begin{split}
& \mb{P} \bigg( \bigcup_{t\in[0,\gamma_1 n/2k]} \bigcup_{r\in\{1,\ldots,k\}} \Big\{ \Big| \msf{X}^\mr{left}_{2kt}-\msf{X}_{(2t+1)k+r}^\mr{left} - \frac{k-r}{2} \Big| > \kappa_2\sqrt{k\log n} \Big\} \ \bigg| \ X_0 \bigg) \\ &\leq \mb{P}\bigg(\bigcup_{\substack{t\in[0,\gamma_1 n/2k] \\ r\in\{1,\ldots,k\}}} \Big\{ \Big| \msf{X}_{2kt}^\mr{left}-\msf{X}_{2kt+r}^\mr{left}-\frac{r}{2} \Big| \vee \Big| \msf{X}_{2kt}^\mr{right}-\msf{X}^\mr{right}_{(2t+1)k+r}-\frac{r}{2} \Big| > \frac{1}{2}\kappa_2\sqrt{k\log n} \Big\} \bigg| X_0 \bigg) \stackrel{\eqref{eq:hoeffding}}{\lesssim} \frac{\kappa_1^2 e^{2\gamma_1}}{\kappa_2},
\end{split}
\end{equation*}
which concludes the proof of the lemma.
\end{proof}

Combining Lemma \ref{lem:hitting} with Lemma \ref{lem:hoeffding}, we deduce the following proposition.

\begin{proposition} \label{prop:tau2}
Let $(X_t), (Y_t)$ be two independent copies of the Bernoulli--Laplace chain with $k=o(n)$ swaps such that $X_0, Y_0 \in\big(\frac{n}{2}-\kappa_1\sqrt{n}, \frac{n}{2}+\kappa_1\sqrt{n}\big)$ for some $\kappa_1\in(0,\infty)$. For $\kappa_2\in(0,\infty)$, consider the stopping time
\begin{equation} \label{eq:tau2}
\tau_2(\kappa_2) \eqdef \min\Big\{t: \ |X_t-Y_t| \leq 2\kappa_2\sqrt{k\log n} \ \ \mbox{and} \ \ X_t, Y_t \in\Big(\frac{n}{2}-\kappa_2\sqrt{n}, \frac{n}{2}+\kappa_2\sqrt{n} \Big) \Big\}.
\end{equation} 
Then, for every $\gamma_1\in(0,\infty)$, we have
\begin{equation} \label{eq:boundtau2}
\mb{P}\Big(\tau_2(\kappa_2) > \frac{\gamma_1 n}{k} \ \Big| \ X_0, Y_0\Big) \lesssim \frac{\kappa_1}{\sqrt{\gamma_1}} + \frac{\kappa_1^2e^{2\gamma_1}}{\kappa_2}.
\end{equation}
\end{proposition}

\begin{proof}
Consider the events $E_{\kappa_2,\gamma_1}$ and $F_{\kappa_2,\gamma_1}$ of \eqref{eq:theeventE} and \eqref{eq:Fevent} as well as
$$G_{\kappa_2,\gamma_1} \eqdef \bigcap_{t\in[0,\gamma_1 n/2k]} \bigcap_{r\in\{1,\ldots,k\}} \Big\{ \Big| \msf{X}^\mr{left}_{2kt}-\msf{X}_{(2t+1)k+r}^\mr{left} - \frac{k-r}{2} \Big| \leq \kappa_2\sqrt{k\log n} \Big\}$$
and
$$H_{\gamma_1} \eqdef \big\{ \tau_\mr{match} \leq \gamma_1n\big\}.$$
Then, by equations \eqref{eq:doobgoodstart}, \eqref{eq:lemtaumatch}, \eqref{eq:hoeffding} and \eqref{eq:hoeffding2}, we have
\begin{equation} \label{eq:useeverything}
\mb{P}\big( E_{\kappa_2,\gamma_1}^\msf{c} \cup F_{\kappa_2,\gamma_1}^\msf{c}\cup G_{\kappa_2,\gamma_1}^\msf{c} \cup H_{\gamma_1}^\msf{c} \ \big| \ X_0, Y_0 \big) \lesssim \frac{\kappa_1}{\sqrt{\gamma_1}} + \frac{\kappa_1^2 e^{2\gamma_1}}{\kappa_2}.
\end{equation}
We now claim that
\begin{equation} \label{eq:maininclusion}
E_{\kappa_2,\gamma_1}\cap F_{\kappa_2,\gamma_1} \cap G_{\kappa_2,\gamma_1} \cap H_{\gamma_1} \subseteq \Big\{ \tau_2(\kappa_2) \leq \frac{\gamma_1 n}{k}\Big\},
\end{equation}
which, along with \eqref{eq:useeverything}, proves \eqref{eq:boundtau2}.

To see this, suppose that the events $E_{\kappa_2,\gamma_1}, F_{\kappa_2,\gamma_1}, G_{\kappa_2,\gamma_1}$ and $H_{\gamma_1}$ occur simultaneously. In particular, $\tau_\mr{match}\leq \gamma_1 n$.  Let $r = \tau_\mr{match}  \ \! (\!\!\!\!\mod2k)$ and notice that if $r=0$, then 
$$X_{\tau_\mr{match}/2k} = \msf{X}_{\tau_\mr{match}}^\mr{left} = \msf{Y}_{\tau_\mr{match}}^\mr{left} = Y_{\tau_\mr{match}/2k},$$ 
which in particular implies that $\tau_2(\kappa_2) \leq \gamma_1n/2k$. We can therefore assume that $r\in\{1,\ldots,2k-1\}$.

First, suppose that $r\in\{1,\ldots,k\}$, in which case $\msf{X}_{\tau_\mr{match}}^\mr{left} = \msf{Y}_{\tau_\mr{match}}^\mr{left}$. Then, since $F_{\kappa_2,\gamma_1}$ occurs and also $\lfloor\tau_\mr{match}/2k\rfloor \leq \gamma_1n/2k$, we get
\begin{equation*}
\begin{split}
\big|X_{\lfloor\tau_\mr{match}/2k\rfloor} -Y_{\lfloor\tau_\mr{match}/2k\rfloor}\big| & \leq \big| \msf{X}_{2k\lfloor \tau_\mr{match}/2k\rfloor}^\mr{left} - \msf{X}_{\tau_\mr{match}}^\mr{left} - \frac{r}{2}\big| +\big| \msf{Y}_{2k\lfloor \tau_\mr{match}/2k\rfloor}^\mr{left} -  \msf{Y}_{\tau_\mr{match}}^\mr{left} - \frac{r}{2}\big| \\ & \leq 2\kappa_2 \sqrt{k\log n},
\end{split}
\end{equation*}
which implies that $\tau_2(\kappa_2)\leq \gamma_1 n/2k$.

If $r\in\{k+1,\ldots,2k-1\}$, it could be that either $\msf{X}_{\tau_\mr{match}}^\mr{left} = \msf{Y}_{\tau_\mr{match}}^\mr{left}$ or $\msf{X}_{\tau_\mr{match}}^\mr{right} = \msf{Y}_{\tau_\mr{match}}^\mr{right}$. In the former case, using the fact that $G_{\kappa_2,\gamma_1}$ occurs, we deduce that
\begin{equation*}
\begin{split}
\big|X_{\lfloor\tau_\mr{match}/2k\rfloor} -Y_{\lfloor\tau_\mr{match}/2k\rfloor}\big| & \leq \big| \msf{X}_{2k\lfloor \tau_\mr{match}/2k\rfloor}^\mr{left} - \msf{X}_{\tau_\mr{match}}^\mr{left} - \frac{k-r}{2}\big| +\big| \msf{Y}_{2k\lfloor \tau_\mr{match}/2k\rfloor}^\mr{left} - \msf{Y}_{\tau_\mr{match}}^\mr{left} - \frac{k-r}{2}\big| \\ & \leq 2\kappa_2 \sqrt{k\log n}
\end{split}
\end{equation*}
and in the latter case, because $F_{\kappa_2,\gamma_1}$ occurs,
\begin{equation*}
\begin{split}
\big|X_{\lfloor\tau_\mr{match}/2k\rfloor} & -Y_{\lfloor\tau_\mr{match}/2k\rfloor}\big|  = \big| \msf{X}_{2k\lfloor\tau_\mr{match}/2k\rfloor}^\mr{right} - \msf{Y}_{2k\lfloor\tau_\mr{match}/2k\rfloor}^\mr{right}\big| \\ & \leq \big| \msf{X}_{2k\lfloor\tau_\mr{match}/2k\rfloor}^\mr{right} - \msf{X}_{\tau_\mr{match}}^\mr{right} - \frac{r-k}{2} \big| + \big| \msf{Y}_{2k\lfloor\tau_\mr{match}/2k\rfloor}^\mr{right} - \msf{Y}_{\tau_\mr{match}}^\mr{right} - \frac{r-k}{2} \big| \leq 2\kappa_2\sqrt{k\log n}.
\end{split}
\end{equation*}
The above complete the proof of the claim \eqref{eq:maininclusion} and \eqref{eq:boundtau2} follows.
\end{proof}

\begin{remark} \label{rem:weakbounds}
A more straightforward application of Lemma \ref{lem:lpw}, shows that under the assumption that $X_0, Y_0\in\big(\frac{n}{2}-\kappa_1\sqrt{n}, \frac{n}{2}+\kappa_1\sqrt{n}\big)$ for some $\kappa_1\in(0,\infty)$, the difference $|X_t-Y_t|$ will become at most $O(k)$ within $O\big(\frac{n}{k}\big)$ steps. Notice that this improves upon Proposition \ref{prop:tau2} in the range $k\lesssim \log n$. 

We can clearly assume that $k<\sqrt{n}$, since otherwise already $|X_0-Y_0|\lesssim k$. Suppose that $X_0>Y_0$ and consider the stopping time
\begin{equation}
\ms{T} \eqdef \min\{t: \ X_t-Y_t<4k \}.
\end{equation}
Since for every $t\geq0$ we have
$$\|X_{t+1}-X_t\|_\infty = \|Y_{t+1}-Y_t\|_\infty = 2k,$$
it is clear that $X_\ms{T}\geq Y_\ms{T}$ and thus $|X_\ms{T}-Y_\ms{T}|<4k$. Similarly to the proof of Lemma \ref{lem:hitting}, consider the truncated stopping time
\begin{equation}
\widetilde{\ms{T}} \eqdef \ms{T}\wedge \min\Big\{t: \max_{r\in[0,t]}\Big| X_r - \frac{n}{2} \Big|\vee \Big| Y_r-\frac{n}{2}\Big| > \frac{n}{4}\Big\}
\end{equation} 
and notice that, by \eqref{eq:doobgoodstart},
\begin{equation} \label{eq:lemtotilde}
\mb{P}\Big(\ms{T}>\frac{\gamma_1n}{k}\ \Big| \ X_0, Y_0 \Big) \lesssim \mb{P}\Big(\widetilde{\ms{T}}>\frac{\gamma_1n}{k}\ \Big| \ X_0, Y_0\Big) + \frac{\kappa_1^2 e^{4\gamma_1}}{n}.
\end{equation}
Now let $W_t = X_t- Y_t$ and $Z_t = W_{t\wedge\ms{T}}$. It follows from the discussion above that $Z_t\geq0$ for every $t\geq0$ and furthermore, if $t<\ms{T}$, then 
$$\mb{E}\big(W_{t+1} \ \big| \ \ms{F}_t\big) \stackrel{\eqref{eq:eigenv1}}{=} \Big(1-\frac{2k}{n}\Big)\big(X_t-Y_t) < X_t-Y_t = W_t,$$
where $\ms{F}_t = \sigma\big( X_r, Y_r\big)_{r\leq t}$. In other words, $(Z_t)$ is a non-negative supermartingale. As in the proof of Lemma \ref{lem:hitting}, a straightforward calculation involving the eigenvalues of the chain shows that for $t<\widetilde{\ms{T}}$, we have
$$\mr{Var}\big(W_{t+1} \ \big| \ \ms{F}_t\big) \gtrsim k.$$
Therefore, Lemma \ref{lem:lpw} applied to the stopping time $\widetilde{\ms{T}}$ shows that for $u\gtrsim k$, we have
$$\mb{P}\big( \widetilde{\ms{T}} >u \ \Big| \ X_0, Y_0\big)\lesssim \frac{\kappa_1 \sqrt{n}}{\sqrt{ku}}.$$
Thus, since $\frac{n}{k}\geq k$, we deduce that for $\gamma_1\in(0,\infty)$,
$$\mb{P}\big( \widetilde{\ms{T}} >\frac{\gamma_1 n}{k} \ \Big| \ X_0,Y_0\Big) \lesssim \frac{\kappa_1}{\sqrt{\gamma_1}},$$
which combined \eqref{eq:lemtotilde} proves the claim. In Remark \ref{rem:easycutoff} below, we explain why this observation is sufficient to show that the Bernoulli--Laplace chain with $k$ swaps exhibits cutoff for $k=n^{o(1)}$ but it appears that the more delicate treatment of Lemma \ref{lem:hitting} and Lemma \ref{lem:hoeffding} is needed for asymptotically larger values of $k$, e.g., to preserve the non-negativity of the supermartingale $(Z_t)$.
\end{remark}


\subsection{Getting at distance $o\big(\sqrt{k}\big)$} \label{subsec:3.3} Let $(X_t)$ and $(Y_t)$ be two instances of the Bernoulli--Laplace chain with $k=o(n)$ swaps. In this subsection we will show that if $(X_t)$ and $(Y_t)$ are coupled according the coupling of Subsection \ref{subsec:2.1} after the stopping time $\tau_2(\kappa_2)$ of Proposition \ref{prop:tau2} then with high probability the difference $|X_t-Y_t|$ will become asymptotically smaller than $\sqrt{k}$ within $\frac{n}{k}\log\log n$ steps.

\begin{lemma} \label{lem:tau3}
Let $(X_t)$ and $(Y_t)$ be two instances of the Bernoulli--Laplace chain with $k$ swaps, coupled as in Subsection \ref{subsec:2.1}, such that $X_0, Y_0\in\big(\frac{n}{2} - \kappa_2\sqrt{n},\frac{n}{2}+\kappa_2\sqrt{n}\big)$ and $|X_0-Y_0|\leq \kappa_2\sqrt{k\log n}$ for some $\kappa_2\in(0,\infty)$. For $\kappa_3\in(0,\infty)$, consider the stopping time
\begin{equation} \label{eq:tau3}
\tau_3(\kappa_3) \eqdef \min\Big\{t: \ |X_t-Y_t|\leq \frac{\sqrt{k}}{\log\log n} \ \ \mbox{and} \ \ X_t, Y_t\in\Big(\frac{n}{2}-\kappa_3\sqrt{n}(\log n)^2, \frac{n}{2}+\kappa_3\sqrt{n}(\log n)^2\Big)\Big\}.
\end{equation}
Then, we have
\begin{equation} \label{eq:boundtau3}
\mb{P}\Big(\tau_3(\kappa_3)>\frac{n}{k}\log\log n \ \Big| \ X_0,Y_0\Big) \lesssim \frac{\kappa_2^2}{\kappa_3^2}.
\end{equation}
\end{lemma}

\begin{proof}
Recall that we assume that $k\leq\frac{n}{2}$. Following the notation \eqref{eq:taur} of Proposition \ref{prop:coupling}, let
$$\tau_\mr{couple}\Big(\frac{\sqrt{k}}{\log\log n}\Big) \eqdef \min\Big\{t: \ |X_t-Y_t|\leq\frac{\sqrt{k}}{\log\log n}\Big\}.$$
Then, by  \eqref{eq:lemcoupling}, we know that
\begin{equation} \label{eq:boundthetaucouple}
\begin{split}
\mb{P}\Big(\tau_\mr{couple}\Big(\frac{\sqrt{k}}{\log\log n}\Big)>\frac{n}{k}\log\log n  \Big| X_0, & Y_0 \Big)  \leq\Big(1-\frac{2k(n-k)}{n^2}\Big)^{\frac{n}{k}\log\log n} \kappa_2\sqrt{\log n} \log\log n \\ & \leq \Big(1-\frac{k}{n}\Big)^{\frac{n}{k}\log\log n} \kappa_2 \sqrt{\log n} \log\log n \leq \frac{\kappa_2 \log\log n}{\sqrt{\log n}}.
\end{split}
\end{equation}
Also, notice that by the definition of $\tau_3(\kappa_3)$, we have
\begin{equation*} \label{eq:breaktau3}
\begin{split}
\Big\{ \tau_3(\kappa_3) > \frac{n}{k}\log\log n\Big\} \subseteq \Big\{\tau_\mr{couple}\Big(\frac{\sqrt{k}}{\log\log n}\Big) & > \frac{n}{k}\log\log n\Big\} \\ & \cup \bigcup_{t\in[0,\frac{n}{k}\log\log n]}\!\!\! \Big\{ \Big| X_t-\frac{n}{2}\Big|\vee\Big| Y_t-\frac{n}{2}\Big| > \kappa_3\sqrt{n}(\log n)^2 \Big\}.
\end{split}
\end{equation*}
Finally, by \eqref{eq:loglog}, we derive the estimate
$$\mb{P}\Big( \bigcup_{t\in[0,\frac{n}{k}\log\log n]}\!\!\! \Big\{ \Big| X_t-\frac{n}{2}\Big|\vee\Big| Y_t-\frac{n}{2}\Big| > \kappa_3\sqrt{n}(\log n)^2 \Big\}  \ \Big| \ X_0, Y_0\Big) \lesssim \frac{\kappa_2^2}{\kappa_3^2},$$
which along with the last inclusion and \eqref{eq:boundthetaucouple} imply \eqref{eq:boundtau3}.
\end{proof}

For technical purposes, it will be important to know that $|X_t-Y_t| =o(\sqrt{k})$ and $\big|X_t-\frac{n}{2}\big|\vee \big|Y_t-\frac{n}{2}\big| = O(\sqrt{n})$ simultaneously, instead of the weaker bound $\big|X_t-\frac{n}{2}\big|\vee \big|Y_t-\frac{n}{2}\big| = O(\sqrt{n}(\log n)^2)$ which was shown in Lemma \ref{lem:tau3}. This refinement is achieved in the following proposition.

\begin{proposition} \label{prop:tau4}
Let $(X_t), (Y_t)$ be two instances of the Bernoulli--Laplace chain with $k$ swaps, coupled as in Subsection \ref{subsec:2.1} such that $X_0, Y_0\in\big(\frac{n}{2} - \kappa_3\sqrt{n}(\log n)^2,\frac{n}{2}+\kappa_3\sqrt{n}(\log n)^2\big)$ and $|X_0-Y_0|\leq \frac{\sqrt{k}}{\log\log n}$ for some $\kappa_3\in(0,\infty)$. For $\kappa_4\in(0,\infty)$, consider the stopping time
\begin{equation} \label{eq:tau4}
\tau_4(\kappa_4) \eqdef \min\Big\{t: \ |X_t-Y_t|\leq \frac{\sqrt{k}}{\log\log n} \ \ \mbox{and} \ \ X_t, Y_t\in\Big(\frac{n}{2} - \kappa_4\sqrt{n}, \frac{n}{2} + \kappa_4\sqrt{n}\Big) \Big\}.
\end{equation}
Then, we have
\begin{equation} \label{eq:boundtau4}
\mb{P}\Big( \tau_4(\kappa_4)> \frac{2n}{k} \log\log n \ \Big| \ X_0, Y_0\Big) \lesssim \frac{1}{\kappa_4^2}.
\end{equation}
\end{proposition}

\begin{proof}
Observe that, by the monotonicity property \eqref{eq:couplingismonotone} of the coupling, for every $t\geq0$ we have
\begin{equation*}
|X_t-Y_t|\leq \frac{\sqrt{k}}{\log\log n},
\end{equation*}
therefore
\begin{equation} \label{eq:tau4alternative}
\tau_4(\kappa_4) = \min\Big\{t: \ X_t, Y_t\in\Big(\frac{n}{2} - \kappa_4\sqrt{n}, \frac{n}{2} + \kappa_4\sqrt{n}\Big) \Big\}.
\end{equation}
Also, using \eqref{eq:expectationident} we get that for $t\geq\frac{2n}{k}\log\log n$,
\begin{equation}
\big| \mb{E}\big(X_t|X_0\big) - \frac{n}{2}\big| \leq \Big(1-\frac{2k}{n}\Big)^{\frac{2n}{k}\log\log n} \big|X_0-\frac{n}{2}\big| \leq \frac{\kappa_3 \sqrt{n}(\log n)^2}{(\log n)^4} = o\big(\sqrt{n}\big),
\end{equation}
and the same holds for $Y_t$. Combining this with \eqref{eq:tau4alternative} we deduce that
$$\mb{P}\Big(\tau_4(\kappa_4) > \frac{2n}{k}\log\log n \ \Big| \ X_0, Y_0\Big) \leq \mb{P} \Big( \big|X_{t^\ast}-\mb{E}\big(X_{t^\ast}\big| X_0\big) \big| \vee \big|Y_{t^\ast}-\mb{E}\big(Y_{t^\ast}\big| Y_0\big) \big| > \frac{1}{2}\kappa_4 \sqrt{n} \ \Big| \ X_0, Y_0\Big)$$
where $t^* = \frac{2n}{k}\log\log n$. Finally, it immediately follows from the proof of Lemma \ref{lem:expandvar} that also $\mr{Var}\big(X_t\big| X_0\big)\leq n$ for every $t\geq0$, thus by Chebyshev's inequality
$$\mb{P}\Big( \big| X_{t^\ast}-\mb{E}\big(X_{t^\ast}\big| X_0\big)\big| > \frac{1}{2}\kappa_4\sqrt{n} \ \Big| \ X_0\Big) \lesssim \frac{\mr{Var}(X_{t^\ast}|X_0)}{\kappa_4^2 n} \leq \frac{1}{\kappa_4^2},$$
and similarly for $(Y_t)$; thus \eqref{eq:boundtau4} follows.
\end{proof}

\begin{remark} \label{rem:easycutoff}
Recall that in Remark \ref{rem:weakbounds}, it was shown that if $X_0, Y_0 \in\big(\frac{n}{2}-\kappa_1\sqrt{n}, \frac{n}{2}+\kappa_1\sqrt{n}\big)$ for some $\kappa_1\in(0,\infty)$, then within $O\big(\frac{n}{k}\big)$ steps, the difference $|X_t-Y_t|$ will become at most $O(k)$. An argument identical to that of Lemma \ref{lem:tau3} shows that if $(X_t)$ and $(Y_t)$ are coupled according to the coupling of Subsection \ref{subsec:2.1} and $|X_0-Y_0|\lesssim k$, then within $O\big(\frac{n}{k}\log k\big)$ steps $t$, we will have $X_t=Y_t$ with high probability. By standard considerations (see, e.g. \cite[Chapter~5]{LPW}), this implies that the mixing time of the Bernoulli--Laplace chain with $k$ swaps is
\begin{equation}
t_{\mr{mix}} \leq \frac{n}{4k}\log n + O\Big(\frac{n}{k}\log k\Big).
\end{equation}
Therefore, this simple argument is sufficient to prove that the chain exhibits cutoff at $\frac{n}{4k}\log n$ with window of size $\frac{n}{k}\log k$ when $\log k = o(\log n)$ or, in other words, $k=n^{o(1)}$. Notice that the size of the window given by this approach is better than the one claimed in Theorem \ref{main} when $k \leq (\log n)^{O(1)}$.
\end{remark}


\subsection{Proof of Theorem \ref{main}} We are now in position to complete the proof of the main result of this article, namely Theorem \ref{main}. The simple main idea is the following. Fix $\kappa_4\in(0,\infty)$. If a starting point $X_0\in\big(\frac{n}{2}-\kappa_4\sqrt{n}, \frac{n}{2}+\kappa_4\sqrt{n}\big)$ is fixed, then the distribution of $X_1$ is approximately a Gaussian with mean $X_0 + o(\sqrt{k})$ and variance of the order of $k$. Similarly, under the same condition for $Y_0$, the distribution of $Y_1$ is approximately a Gaussian with mean $Y_0+o(\sqrt{k})$ and variance proportional to $k$. Therefore, if $|X_0-Y_0| = o(\sqrt{k})$, then $X_1$ and $Y_1$ have to be $o(1)$-close in total variation distance (see also Figures \ref{fig1} and \ref{fig2} below). To make this intuition precise, we first show a lemma based on the work \cite{DiaFree} of Diaconis and Freedman, where they studied the total variation distance of sampling with and without replacement.

\begin{lemma} \label{lem:diafree}
Let $X_0, Y_0 \in \big( \frac{n}{2}-\kappa_4\sqrt{n}, \frac{n}{2}+\kappa_4\sqrt{n}\big)$ for some $\kappa_4\in(0,\infty)$. Then,
\begin{equation} \label{eq:diafree}
\|X_1-Y_1\|_{\mr{TV}} \lesssim \|B_1-B_2\|_{\mr{TV}} + \kappa_4\sqrt{\frac{k}{n}},
\end{equation}
where $B_1\sim \mr{Bin}\big(k,\frac{1}{2}\big)$ and $B_2 - (X_0-Y_0) \sim \mr{Bin}\big(k,\frac{1}{2}\big)$.
\end{lemma}

\begin{proof}
By the definition of the Bernoulli--Laplace chain, we have that
\begin{equation} \label{eq:diafree1}
X_1-X_0 = H_1 - H_2,
\end{equation}
where $H_1\sim\mr{Hyper}(n,n-X_0,k)$ is the number of red balls removed from the right urn and $H_2\sim\mr{Hyper}(n,X_0,k)$ is the number of red balls removed from the left urn and $H_2$ is independent of $H_1$. Similarly, $Y_1-Y_0 \sim H_3-H_4$, where $H_3\sim\mr{Hyper}(n,n-Y_0,k)$ and $H_4\sim\mr{Hyper}(n,Y_0,k)$ is independent of $H_3$. Consider binomial distributions $M_1 \sim \mr{Bin}\big(k, \frac{X_0}{n}\big)$, $M_2\sim\mr{Bin}\big(k,1-\frac{X_0}{n}\big)$, $M_3 \sim \mr{Bin}\big(k, \frac{Y_0}{n}\big)$ and $M_4\sim\mr{Bin}\big(k,1-\frac{Y_0}{n}\big)$. According to \cite[Theorem~(3)]{DiaFree}, for every $i\in\{1,2,3,4\}$,
\begin{equation} \label{eq:diafree2}
\|H_i - M_i\|_{\mr{TV}} \leq \frac{4k}{n}.
\end{equation}
Moreover, if $B\sim\mr{Bin}\big(k,\frac{1}{2}\big)$, it is well-known (see, e.g., \cite{Roos}) that
\begin{equation} \label{eq:diafree3}
\|M_1-B\|_{\mr{TV}} \lesssim \sqrt{k} \Big|\frac{1}{2}-\frac{X_0}{n}\Big| \lesssim \kappa_4 \sqrt{\frac{k}{n}},
\end{equation}
since $X_0\in\big(\frac{n}{2}-\kappa_4\sqrt{n}, \frac{n}{2}+\kappa_4\sqrt{n}\big)$ and the same holds for $M_2, M_3, M_4$. Combining \eqref{eq:diafree1}, \eqref{eq:diafree2} and \eqref{eq:diafree3}, we deduce that if $B_1\sim \mr{Bin}\big(k,\frac{1}{2}\big)$ and $B_2-(X_0-Y_0)\sim\mr{Bin}\big(k,\frac{1}{2}\big)$, then
$$\|X_1-Y_1\|_{\mr{TV}} \lesssim \|B_1-B_2\|_{\mr{TV}} +\frac{k}{n}+\kappa_4 \sqrt{\frac{k}{n}} \lesssim \|B_1-B_2\|_{\mr{TV}} + \kappa_4\sqrt{\frac{k}{n}},$$
as we wanted.
\end{proof}

\begin{center}
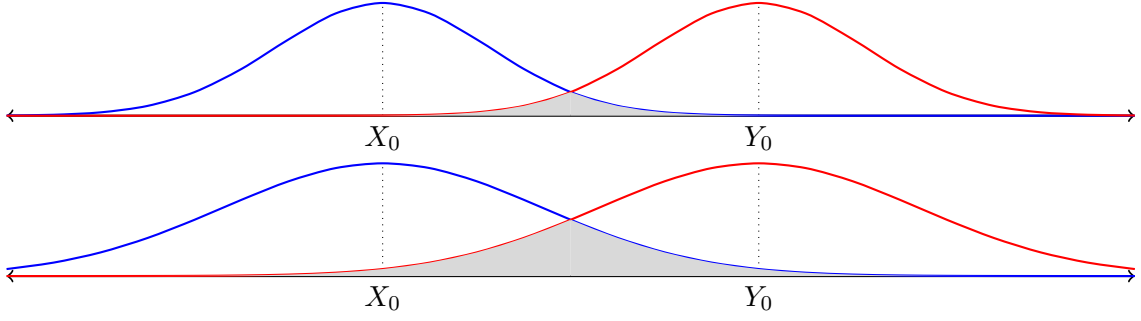
\begin{figure}
\begin{tikzpicture}
\draw[thick][<->] (0,0) -- (15,0) node[anchor=north] {};
\draw	(5,0) node[anchor=north] {$X_0$}
		(10,0) node[anchor=north] {$Y_0$};
\draw[dotted] (5,0) -- (5,1.5);
\draw[dotted] (10,0) -- (10,1.5);
    \draw[thick][scale=1,domain=0:15,smooth,variable=\x,blue] plot ({\x},{1.5*pow(e,-(\x-5)*(\x-5)/4)});
    \draw[thick][scale=1,domain=0:15,smooth,variable=\x,red] plot ({\x},{1.5*pow(e,-(\x-10)*(\x-10)/4)});

    \fill [gray!30, domain=7.5:15, variable=\x]
      (7.5, 0)
      -- plot ({\x}, {1.5*pow(e,-(\x-5)*(\x-5)/4)})
     -- (15, 0)
      -- cycle;

    \fill [gray!30, domain=0:7.5, variable=\x]
      (0, 0)
      -- plot ({\x}, {1.5*pow(e,-(\x-10)*(\x-10)/4)})
      -- (7.5, 0)
      -- cycle;

\end{tikzpicture}

\begin{tikzpicture}
\draw[thick][<->] (0,0) -- (15,0) node[anchor=north] {};
\draw	(5,0) node[anchor=north] {$X_0$}
		(10,0) node[anchor=north] {$Y_0$};
\draw[dotted] (5,0) -- (5,1.5);
\draw[dotted] (10,0) -- (10,1.5);
    \draw[thick][scale=1,domain=0:15,smooth,variable=\x,blue] plot ({\x},{1.5*pow(e,-(\x-5)*(\x-5)/9)});
    \draw[thick][scale=1,domain=0:15,smooth,variable=\x,red] plot ({\x},{1.5*pow(e,-(\x-10)*(\x-10)/9)});

    \fill [gray!30, domain=7.5:15, variable=\x]
      (7.5, 0)
      -- plot ({\x}, {1.5*pow(e,-(\x-5)*(\x-5)/9)})
      -- (15, 0)
      -- cycle;

    \fill [gray!30, domain=0:7.5, variable=\x]
      (0, 0)
      -- plot ({\x}, {1.5*pow(e,-(\x-10)*(\x-10)/9)})
      -- (7.5, 0)
      -- cycle;

\end{tikzpicture}
\caption{The approximate distributions of $X_t$ and $Y_t$ for $t=1,2$ if $|X_0-Y_0|\asymp\sqrt{k}$} \label{fig1}
\end{figure}
\begin{figure}
\begin{tikzpicture}
\draw[thick][<->] (0,0) -- (15,0) node[anchor=north] {};
\draw	(7,0) node[anchor=north] {$X_0$}
		(8,0) node[anchor=north] {$Y_0$};

    \draw[thick][scale=1,domain=0:15,smooth,variable=\x,blue] plot ({\x},{1.5*pow(e,-(\x-7)*(\x-7)/16)});
    \draw[thick][scale=1,domain=0:15,smooth,variable=\x,red] plot ({\x},{1.5*pow(e,-(\x-8)*(\x-8)/16)});

    \fill [gray!30, domain=7.5:15, variable=\x]
      (7.5, 0)
      -- plot ({\x}, {1.5*pow(e,-(\x-7)*(\x-7)/16)})
      -- (15, 0)
      -- cycle;

    \fill [gray!30, domain=0:7.5, variable=\x]
      (0, 0)
      -- plot ({\x}, {1.5*pow(e,-(\x-8)*(\x-8)/16)})
      -- (7.5, 0)
      -- cycle;

\draw[dotted] (7,0) -- (7,1.5);
\draw[dotted] (8,0) -- (8,1.5);
\end{tikzpicture}
\caption{The approximate distributions of $X_1$ and $Y_1$ if $|X_0-Y_0|=o\big(\sqrt{k}\big)$}\label{fig2}
\end{figure}
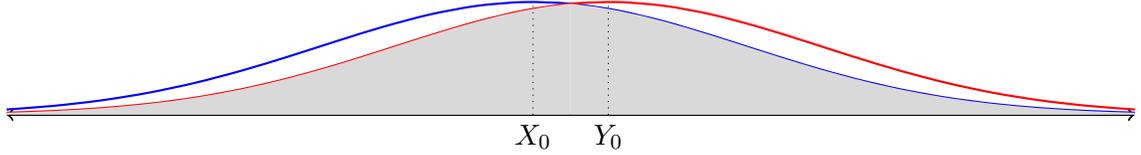
\end{center}

We will also need the following simple computation.

\begin{lemma} \label{lem:binom}
Let $B_1, B_2$ be two random variables such that $B_1 \sim \mr{Bin}\big(k,\frac{1}{2}\big)$ and $B_2-g(k) \sim\mr{Bin}\big(k,\frac{1}{2}\big)$ for some function $g(k)\in\mathbb{Z}$ with $g(k) = o\big(\sqrt{k}\big)$ as $k\to\infty$. Then,
\begin{equation}
\|B_1-B_2\|_{\mr{TV}} = o(1), \ \ \ \mbox{as} \ k\to\infty.
\end{equation}
\end{lemma}

\begin{proof}
Let 
$$\mu(x) \eqdef \binom{k}{x} \frac{1}{2^k}, \ \ x\in\{0,1,\ldots,k\}$$ 
be the law of $B_1$ and 
$$\nu(x)\eqdef \binom{k}{x-g(k)}\frac{1}{2^k}, \ \ x\in\{g(k),g(k)+1,\ldots,g(k)+k\}$$ 
be the law of $B_2$. A straightforward computation shows that there exists a point $x^\ast\in\{g(k),\ldots,k\}$ such that $\mu(x)\geq\nu(x)$ if and only if $x\leq x^\ast$. Therefore,
\begin{equation*}
\begin{split}
\|B_1-&B_2\|_{\mr{TV}} = \sum_{x=0}^{x^\ast} \big(\mu(x) - \nu(x)\big) = \mb{P}\Big(x^\ast-g(k)<\mr{Bin}\Big(k,\frac{1}{2}\Big) \leq x^\ast\Big) \\ & \leq \mb{P}\Big(-\frac{g(k)}{2}<\mr{Bin}\Big(k,\frac{1}{2}\Big)-\frac{k}{2} \leq \frac{g(k)}{2}\Big) = \mb{P}\Big(-\frac{g(k)}{\sqrt{k}}<\frac{\mr{Bin}\big(k,\frac{1}{2}\big)-\frac{k}{2}}{\sqrt{k}/2} \leq \frac{g(k)}{\sqrt{k}}\Big) = o(1),
\end{split}
\end{equation*}
where the first inequality follows from the unimodality of $\mr{Bin}\big(k,\frac{1}{2}\big)$ and the last equality follows from the central limit theorem, since $g(k)=o(\sqrt{k})$ as $k\to\infty$.
\end{proof}

\noindent{\it Proof of Theorem \ref{main}.}
We can clearly assume that $k\to\infty$ as $n\to\infty$ since the case $k=O(1)$ is covered by Remarks \ref{rem:weakbounds} and \ref{rem:easycutoff}. Let $(X_t)$ and $(Y_t)$ be two instances of the Bernoulli--Laplace chain with $X_0\in\{0,1,\ldots,n\}$ being fixed and $Y_0\sim \pi_n$. Combining Lemma \ref{lem:usevariance}, Proposition \ref{prop:tau2}, Lemma \ref{lem:tau3} and Proposition \ref{prop:tau4}, we see that if $\tau_4(\kappa_4)$ is defined by \eqref{eq:tau4}, then for every $\kappa_1,\kappa_2,\kappa_3, \kappa_4\in(0,\infty)$,
\begin{equation} \label{eq:usetau1-4}
\mb{P}\Big( \tau_4(\kappa_4)> \frac{n}{4k}\log n + \frac{3n}{k} \log\log n + \frac{\gamma_1 n}{k} \ \Big| \ X_0, Y_0\Big) \lesssim \frac{1}{\kappa_1^2} + \frac{\kappa_1}{\sqrt{\gamma_1}} + \frac{\kappa_1^2e^{2\gamma_1}}{\kappa_2} + \frac{\kappa_2^2}{\kappa_3^2} + \frac{1}{\kappa_4^2}.
\end{equation}
Choosing $\kappa_1 \asymp \gamma_1^{1/4}$, $\kappa_2 \asymp \kappa_1^2 e^{3\gamma_1}$, $\kappa_3 \asymp \kappa_2 e^{\gamma_1}$ and $\kappa_4\asymp\gamma_1$, we finally get
\begin{equation} \label{eq:chooseparameters}
\mb{P}\Big( \tau_4(\gamma_1)> \frac{n}{4k}\log n + \frac{3n}{k} \log\log n + \frac{\gamma_1 n}{k} \ \Big| \ X_0, Y_0\Big) \lesssim \frac{1}{\sqrt{\gamma_1}}+\frac{1}{\sqrt[4]{\gamma_1}} + \frac{1}{e^{\gamma_1}} + \frac{1}{\gamma_1^2} \lesssim \frac{1}{\sqrt[4]{\gamma_1}}.
\end{equation}
Moreover, Lemmas \ref{lem:diafree} and \ref{lem:binom} imply that
\begin{equation} \label{eq:o(1)}
\|X_{\tau_4(\gamma_1)+1} - Y_{\tau_4(\gamma_1)+1}\|_{\mr{TV}} = o(1)
\end{equation}
as $n\to\infty$, therefore a combination of \eqref{eq:chooseparameters} and \eqref{eq:o(1)} yields that for every $x\in\{0,1,\ldots,n\}$,
\begin{equation}
t\geq\frac{n}{4k}\log n + \frac{3n}{k}\log\log n + \frac{\gamma_1 n}{k} \ \ \ \Longrightarrow \ \ \ \big\|P_t^x - \pi_n \big\|_{\mr{TV}} \lesssim \frac{1}{\sqrt[4]{\gamma_1}},
\end{equation}
or, $t_{\mr{mix}}(\e) \leq \frac{n}{4k}\log n+ \frac{3n}{k}\log\log n + O\big(\frac{n}{\e^4k}\big)$, which completes the proof of Theorem \ref{main}.
\hfill$\Box$

\smallskip

\subsection*{Acknowledgements.} We are very grateful to Allan Sly for many helpful discussions.


\bibliographystyle{alpha}
\bibliography{cutoffbl} 

\end{document}